\documentclass[a4paper]{amsart}

\usepackage[body={140mm,230mm}, footskip=9mm]{geometry}
\usepackage{amsmath}
\usepackage{amssymb}
\usepackage{amsthm}
\usepackage{amsfonts}
\usepackage{amstext}
\usepackage{amsopn}
\usepackage{mathrsfs}
\usepackage{color}      % to display colored text
\usepackage{verbatim}   % for "comment" environment
\usepackage{booktabs}

\usepackage{amsmath}
\usepackage{amssymb}
\usepackage{amsfonts}
\usepackage{amstext}
\usepackage{amsopn}
\usepackage{amsxtra}
\usepackage{mathrsfs}
\usepackage{enumerate}
\usepackage[dvips]{graphicx} 
\usepackage{color}      % to display colored text
\usepackage{verbatim}   % for "comment" environment
\usepackage{booktabs}

\newtheorem{lemma}{Lemma}[section]
\newtheorem{theorem}{Theorem}

\theoremstyle{definition}
\newtheorem{definition}[lemma]{Definition}
\theoremstyle{definition}
\newtheorem{remark}[lemma]{Remark}
\theoremstyle{definition}

{\catcode `\@=11 \global\let\AddToReset=\@addtoreset}
\AddToReset{equation}{section}

\begin{document}


%\TitleLanguage[EN]
\title[Beam Stability and Simulation]{A piezoelectric Euler-Bernoulli beam  
with dynamic boundary control: stability and dissipative FEM}

\thanks{The authors were supported by the doctoral school \emph{PDE-Tech} 
of TU Wien and the FWF-project I395-N16. The authors acknowledge a sponsorship by \emph{Clear Sky Ventures}. 
We are grateful to 
A.~Kugi and T.~Meurer for introducing us to this topic, their help, and the many
stimulating discussions.}

\author[M. Miletic]{Maja Miletic} \address{Institute for Analysis and
 Scientific Computing, Technical University Vienna, Wiedner
 Hauptstr. 8, A-1040 Vienna, AUSTRIA}
\email{mmiletic@asc.tuwien.ac.at, anton.arnold@tuwien.ac.at}

\author[A. Arnold]{Anton Arnold}

\begin{abstract}
We present a mathematical and numerical analysis on a control model for the time evolution 
of a multi-layered piezoelectric cantilever with tip mass and moment of inertia, as developed by Kugi and Thull \cite{Kugi:Thull}. 
This closed-loop control system consists of the inhomogeneous Euler-Bernoulli beam equation coupled to an ODE system that is designed to track both the position and angle of the tip mass for a given reference trajectory.
This dynamic controller only employs first order spatial derivatives, in order to make the system technically realizable with piezoelectric sensors.
{}From the literature it is known that it is asymptotically stable \cite{Kugi:Thull}. But in a 
refined analysis we first prove that this system is \emph{not exponentially} stable. 

In the second part of this paper, we construct a dissipative finite element method, based on piecewise cubic Hermitian 
shape functions and a Crank-Nicolson time discretization. For both the spatial semi-discretization and the full $x-t$--discretization we prove that the numerical method is structure preserving, i.e.~it dissipates energy,
analogous to the continuous case. Finally, we derive error bounds for both cases and illustrate the predicted convergence rates in a simulation example.
\end{abstract}

\maketitle                 

%%%%%%%%%%%%%%%%%%%%%%%%%%%%%%%%%%%%%%%%%%%%%%%%%%%%%%%%%%%%%%%%%%%

\section{Model}

The Euler-Bernoulli beam (EBB) equation with tip mass is a well-established model with a wide range of applications: for oscillations in telecommunication antennas, or satellites with flexible appendages \cite{Bal-Tay,Banks:Rosen:2},  flexible wings of micro air vehicles \cite{Chakravarthy:Evans:Evers}, and even vibrations of tall buildings due to external forces \cite{Prosper}. 
The interest of engineers and mathematicians in the corresponding control problems started in the 1980s.
So various boundary control laws have been devised and mathematically analyzed in the literature -- with the stabilization of the system being a key objective (cf.~\cite{Lit-Mar}). 
Soon afterwards, also exponentially stable controllers were developed which require, however, higher order boundary controls for an EBB with both applied tip mass and moment of inertia  \cite{Rao}. On the other hand, if only a tip mass is applied, lower order controls are sufficient for exponential stabilization \cite{ChW2007}.
%\cite{Chen:Delfour:Krall:Payre, Chen:Krantz:Ma:Wayne:West}, where frequency domain criteria have been 
%used to show exponential stability of the system.
In spite of this progress, and due to its widespread technological applications, considerable research on EBB-control problems is still underway:
In the more recent papers \cite{Guo-Wan,Guo2002} exponential stability of related control systems was established by verifying the Riesz basis property.
For the exponential stability of a more general class of boundary control systems (including the Timoshenko beam) in the port-Hamiltonian approach we refer to \cite{VZLGM2009}.

%\cite{Rao} state space was not proper there

\begin{comment}
is in many application fields a standard model for flexible body dynamics;
vibrations of buildings \cite{Prosper}, 
%railway structures \cite{Tao:Wang}, 
and flexible wings of micro air vehicles
\cite{Chakravarthy:Evans:Evers} due to external forces, as well as telecommunication antennas, or
satellites with flexible appendages \cite{Banks:Rosen:2}.
These examples show that, in many areas of industry and technology, the stabilization 
and boundary control of EBB is still an important problem.  Among the first articles dealing with the 
construction and analysis of boundary control for an Euler-Bernoulli beam are  \cite{Chen:Delfour:Krall:Payre} 
and \cite{Chen:Krantz:Ma:Wayne:West}, where frequency domain criteria have been 
used to show exponential stability of the system.
\end{comment}

We shall analyze an inhomogeneous multi-layered piezoelectric EBB with applied tip mass and moment of inertia, 
coupled to a dynamic controller that uses only low order boundary measurements.
This system was introduced by Kugi and Thull in 
\cite{Kugi:Thull} to independently control the tip position and the tip angle of a piezoelectric cantilever 
along prescribed trajectories. 
This beam is composed of piezoelectric layers
and the electrode shape of the layers was used as an additional degree of 
freedom in the controller design. The sensor layers were given rectangular and
triangular shaped electrodes, so that the charge measured is proportional to the 
tip deflection and the tip angle, respectively. The actuator layers were also assumed
to be covered with rectangular and triangular shaped electrodes, with the following 
motivation: A voltage
supplied to an actuator with rectangular (or triangular) shaped electrodes acts in the
same way on the structure as a bending moment (or force) at the tip of the beam.
The key issue of \cite{Kugi:Thull} was to devise a stable feedback control model 
for that beam, such that it evolves asymptotically (as $t\to\infty$) as a prescribed reference trajectory.
More precisely, that controller allows to track the position and the angle of the tip mass at the same time.
To solve the trajectory planning task, the concept of differential flatness (cf.~\cite{Balas}) 
was employed. Thereby, the control inputs and the beam bending deflection were
parametrized by the flat outputs and their time derivatives. 
%In order to stabilize the system, boundary control is applied. 
The boundary controller constructed there has a dynamic design, 
thus coupling the governing PDEs of the beam with a system of ODEs in the feedback
part. In order to render the system experimentally and technically realizable, it 
is crucial that the controller only involves boundary measurements up to the first 
spatial derivative -- at the (small) price of loosing \emph{exponential} stability 
(as we shall see here below).

The goal of the present paper is first to complete the analysis of \cite{Kugi:Thull}, 
proving that this hybrid system is asymptotically stable but \emph{not} exponentially 
stable. This part is an extension of Rao's analysis \cite{Rao} to dynamic controllers 
and inhomogeneous beams.
In our second, and in fact main part we shall develop and analyze a dissipative finite 
element method (FEM) for the control system. %Moreover, our analysis will be the groundwork 
%for investigating nonlinear controllers for the EBB equation in a subsequent work \cite{AKMS}.

Now we specify the problem under consideration, an inhomogeneous EBB of length 
$L$, clamped at the left end $x=0$, and with tip mass, moment of inertia, and boundary control at $x=L$.
In the following linear system (\ref{model1})--(\ref{model5}), we actually consider the evolution of 
the trajectory error system. So, $u(t,x)$ denotes the deviation of the actual beam deflection from the 
desired reference trajectory. Similarly, $\Theta_{1,2}(t)$ denote the difference between the applied voltages 
to the electrodes of the piezoelectric layers and 
the ones specified by the feedforward controller. 
\begin{eqnarray} \label{model1}
   \mu(x)  u_{tt}  + (\Lambda(x) u_{xx})_{xx} & = & 0, 
 \quad 0 < x < L, \, t > 0,\\  
\label{model2} u (t,0) & = & 0, \quad t>0,\\ 
\label{model3}
  u_{x} (t,0) & = & 0, \quad t>0,  \\  
\label{model4} J   u_{xtt} (t, L) + 
(\Lambda u_{xx}) (t , L) + \Theta_{1}(t) &  = & 0,  \quad t>0 ,\\ 
M  u_{tt} 
\label{model5}
 (t, L) - (\Lambda u_{xx})_x (t , L) + \Theta_{2}(t) &  = & 0, \quad t>0.
\end{eqnarray}
Here, $\mu \in C^4[0,L]$ denotes the linear mass density of the beam and $\Lambda \in C^4[0,L]$ 
is the flexural rigidity of the beam. 
%and it is equal to the product of Young's modulus and the second moment of the beam's cross section area. 
Both functions are assumed to be strictly positive and bounded. 
$M$ and $J$ denote, respectively, the mass and the moment of inertia of the rigid body attached at $x=L$. 
Equation (\ref{model4}) states that the beam bending moment at $x=L$
(i.e.\ $\Lambda(L) u_{xx}(t,L)$) plus the bending moment of the tip body 
(i.e.\ $J u_{xtt} (t, L) $) is balanced by the control input $-\Theta_1$. 
Similarly, (\ref{model5}) describes that the total force at the free end, equal to shear force at the tip 
(i.e.\ $ - (\Lambda u_{xx})_x(t , L)$) plus the tip mass force $ M  u_{tt}$, cancels with the control input $\Theta_2$.

The proposed control law has the goal to drive 
the error system to the zero state as $t\to\infty$. It reads:
\begin{equation} \label{control_law}
\begin{array}{rcl} (\zeta_{1})_t(t)  & = & A_{1} \zeta_{1}(t) + 
b_{1}  u_{xt} (t, L), \\  
(\zeta_{2})_t(t)  & = & A_{2} \zeta_{2}(t) + b_{2} u_{t} (t, L), \\  
\Theta_{1} (t) &  = & k_{1}  u_{x} (t, L) + 
c_{1} \cdot \zeta_{1} (t) + d_{1}  u_{xt} (t, L),  \\  
\Theta_{2}(t) &  = & k_{2} u (t, L) + c_{2} \cdot \zeta_{2} (t) + d_{2}  u_{t} (t, L), 
\end{array}
\end{equation}
with the auxiliary variables $\zeta_1, \zeta_2 \in C([0, \infty ) ; \mathbb{R}^n)$
 and $\Theta_1, \Theta_2 \in C[0, \infty)$. Moreover, $A_1, A_2 \in \mathbb{R}^{n
 \times n}$ are Hurwitz\footnote{A square matrix is called a Hurwitz matrix if 
all its eigenvalues have negative real parts.} matrices, $b_1, b_2, c_1, c_2 
\in \mathbb{R}^n$ 
vectors and $k_1, k_2, d_1, d_2 \in \mathbb{R}$.
We assume that the coefficients $k_1$ and $k_2$ are positive and that
 the transfer functions $\mathcal{G}_j (s) =  (s I - A_j)^{-1} b_j \cdot c_j
 + d_j, \, j=1,2$ satisfy \[ Re({\mathcal{G}}_j (i \omega )) \ge d_j \ge \delta
_j > 0 
\quad \forall \omega \ge 0, \; j= 1, 2\] for some constants 
$\delta_1$ and $\delta_2$. These assumptions imply that the transfer function is
 strictly positive real, or shortly SPR (for its definition we refer to \cite{Khalil}, \cite{Luo:Guo:Morgul}). 
Then, it follows from the Kalman-Yakubovic-Popov
 Lemma (see \cite{Khalil}, \cite{Luo:Guo:Morgul}) that there exist symmetric positive definite
 matrices $P_j$, positive scalars $\epsilon_j$, and vectors $q_j\in \mathbb{R}^n$ such that  
\begin{equation} \label{kyp}
\begin{aligned}
P_j A_j + A_j^{\top} P_j & =  - q_j q_j^{\top} - \epsilon_j P_j, 
 \\ P_j b_j & = c_j - q_j \sqrt{2 (d_j - \delta_j)},
\end{aligned}
\end{equation}
for $j=1,2.$ A SPR controller is defined as a controller with SPR transfer function. 
One motivation for this controller design is the fact that, in the 
finite dimensional case, the feedback interconnection of a passive system with 
a SPR controller yields a stable closed-loop system. 
This principle of passivity based controller 
design was generalized to the trajectory error dynamics of the multi-layered piezoelectric 
cantilever in \cite{Kugi:Thull}. 
 
(\ref{model1})--(\ref{control_law}) constitute a coupled 
PDE--ODE system for the beam deflection $u(x, t)$, the position of its 
tip $u (t, L)$, and its slope $u_x(t, L)$, as well as the two control variables $\zeta_1(t)$, 
$\zeta_2(t)$. The main mathematical difficulty of this system stems from the high order boundary conditions 
(involving both $x$- and $t$- derivatives) which makes the analytical and numerical treatment far from obvious.
 Well-posedness of this system and asymptotic stability of the 
zero state were established in \cite{Kugi:Thull} using semigroup theory on an equivalent first order system (in time), a carefully designed 
Lyapunov functional, and LaSalle's invariance principle.

In \S\ref{s2} we shall prove that this unique steady 
state is \emph{not} exponentially stable. 
%which is rather surprising for a mechanical system. 
Let us compare this result to a similar system studied in \cite{Morgul92} and \S5.3 of \cite{Luo:Guo:Morgul}, 
which also consists of an EBB coupled to a passivity based dynamic boundary control, 
but without the tip mass. Then, that system is exponentially stable.

As an introduction for our dissipative finite 
element method (FEM) in \S\ref{S-FEM}, we shall now briefly review several 
numerical strategies for the EBB from the literature.
In \cite{Tzes:Yurkovic:Langer} the authors propose a conditionally stable, central difference method 
for both the space and time discretization of the EBB equation. Their system models a beam, which has a 
tip mass with moment of inertia on the free end. At the fixed end a boundary control is applied in 
form of a control torque. Due to higher order boundary conditions, fictitious nodes are needed at both boundaries.   
In \cite{Dadfarnia:Jalili:Xian:Dawson} the authors 
consider a damped, translationally cantilevered EBB, with one end clamped into a moving base (as a boundary control) 
and a tip mass with moment of inertia placed at the other. {}For their numerical treatment they considered a finite number  of modes, thus obtaining an ODE system. 
%Also \cite{Kugi:Thull, Kugi:Thull:Meurer} are based on a finite 
%dimensional modal approximation of \eqref{model1}--(\ref{control_law}).
In \cite{Liang:Chen:Guo} the EBB with one free end (without tip mass, but with boundary torque control) was solved 
in the frequency domain: After Laplace transformation in time, the resulting ODEs could be solved explicitly.

The more elaborate approaches are based on FEMs: In \cite{Bar-Yoseph:Fisher:Gottlieb} two space-time spectral element 
 methods are employed to solve a simply supported, nonlinear, modified EBB
subjected to forced lateral vibrations but with no mass attached: There, Hermitian  polynomials, both in space and time, 
lead to strict stability limitations. But a mixed discontinuous 
Galerkin formulation with Hermitian cubic polynomials in space and Lagrangian spectral
 polynomials in time yields an unconditionally stable scheme.  In \cite{Choo:Chung} the authors present 
a semi-discrete (using cubic splines) and fully discrete Galerkin scheme (based on the Crank-Nicolson 
method) for the strongly damped, extensible beam equation with both ends hinged. \cite{Banks:Rosen} 
considers a EBB with tip mass at the free end, yielding a conservative hyperbolic system. 
The authors analyze a cubic B-spline based Galerkin method (including convergence analysis of the spatial semi-discretization) 
and put special emphasis on the subsequent parameter identification problem. 
%Their extended model in \cite{Banks:Rosen:2} involves a viscoelastic damping (in the equation),  hence leading to an abstract parabolic system. 

All these FEMs are for models \emph{without boundary control}. 
Hence, we shall develop here a novel FEM for the mixed boundary control problem (\ref{model1})-(\ref{control_law}). 
%We shall consider here the coupled hyperbolic system (\ref{model1})-(\ref{control_law}), where
There, the damping only appears due to the 
boundary control. Hence, our main focus will be on preserving the correct large-time behavior (i.e.\ dissipativity) in the numerical scheme. 
Our FEM is based on the second order (in time) EBB equation (\ref{model1}) and special care is taken for the boundary coupling to the ODE. 
In time we shall use a Crank-Nicolson discretization, which was also the appropriate approach for the decay of discretized parabolic equations in \cite{Arnold:Unterreiter}.
We remark that the modeling and discretization of boundary control systems as port-Hamiltonian systems also has this flavor of preserving the structure: For a general methodology on this spatial semi-discretization (leading to mixed finite elements) and its application to the telegrapher's equations we refer to \cite{GTSM2004}.

The paper is organized as follows: 
In \S\ref{s2} we first review the analytic setting from \cite{Kugi:Thull} for the EBB with boundary control. While this closed-loop system is asymptotically stable, we prove that it is 
\emph{not} exponentially stable. Towards this analysis we derive the asymptotic behavior of the eigenvalues and eigenfunctions of the coupled system.
In \S\ref{S-FEM} we first discuss the weak formulation of our control system. Then we
 develop an unconditionally stable FEM (along with a Crank-Nicolson scheme in time), which dissipates 
an appropriate energy functional independently of the chosen FEM basis. We shall also derive error estimates 
(second order in space and time) of our scheme.
In the numerical simulations of \S\ref{S-simul} we illustrate the proposed method and 
verify its order of convergence w.r.t. $h$ and $\Delta t$.

%%%%%%%%%%%%%%%%%%%%%%%%%%%%%%%%%%%%%%%%%%%%%%%%%%%%%%%%%%%%%%%%%%%%%%%%%%%%%%%%%%%%%
\section{Non-exponential decay}\label{s2}

First we recall from \cite{Kugi:Thull} the analytical setting for 
(\ref{model1})--(\ref{control_law}) in the framework of semigroup theory. 
%The main analytic difficulty of this model stems from the higher order boundary conditions
%\eqref{model4}, \eqref{model5} and the boundary terms on the r.h.s.\ of (\ref{control_law}). 
%It can be dealt with by introducing $u_{t}(t,L)$, $u_{xt}(t,L)$ 
To cope with the higher order boundary conditions
\eqref{model4}, \eqref{model5} and the boundary terms on the r.h.s.\ of (\ref{control_law}), the terms 
$u_{t}(t,L)$, $u_{xt}(t,L)$ were introduced
as separate variables (following the spirit in earlier works \cite{Lit-Mar,Guo2002}). 
More precisely, $\psi = M v(L)$ is the vertical momentum of the tip mass and $\xi = J v_{x}(L)$ 
its angular momentum, where $v = u_t$ is the velocity of the beam. Hence, we define the Hilbert space 
$$\mathcal{H}  
:= \{ z = (u, v, \zeta_{1}, 
\zeta_{2}, \xi, \psi)^\top \colon u \in \tilde H^{2}_{0}(0, L), v \in L^{2}
 (0,L), \zeta_{1}, \zeta_{2} \in \mathbb{R}^{n}, \xi, \psi \in
 \mathbb{R} \},
$$ 
where $\tilde H^{k}_{0}(0,L) := \{ u \in 
H^{k}(0,L) | \quad u(0)= u_x(0) = 0 \}$, with the inner product 
\begin{eqnarray}
 \langle z, \breve{z} \rangle & := & \frac{1}{2} \int_{0}^{L} 
 \Lambda u_{xx} \breve{u}_{xx} \,dx + \frac{1}{2} \int_{0}^{L} \mu v
 \breve{v} \,dx + \frac{1}{2 J} \xi \breve{\xi} + \frac{1}{2 M} \psi 
\breve{\psi} \nonumber \\ & + & \frac{1}{2} k_{1} u_x(L) 
\breve{u}_x(L) + \frac{1}{2} k_{2} u(L) \breve{u}(L) + \frac{1}{2} 
\zeta_{1}^{\top} P_{1} \breve{\zeta}_{1} + 
\frac{1}{2} \zeta_{2}^{\top} P_{2} \breve{\zeta}_{2}, \nonumber
 \end{eqnarray}
 and $\| z\|_{\mathcal{H}}$ denotes the corresponding norm.
Let $\mathcal{A}$ be a linear operator with the domain 
$$D(\mathcal{A}) =  \{z \in \mathcal{H} \colon 
 u \in \tilde H^{4}_{0}(0, L), v \in \tilde H^{2}_{0} (0,L), \zeta_{1}, 
\zeta_{2} \in \mathbb{R}^{n}, \xi = J v_x(L), \psi = M v(L) \},
$$
%\begin{eqnarray} 
%D(\mathcal{A})& = & \{z = (u, v, \zeta_{1}, \zeta_{2}, \xi, \psi)^\top \colon 
% u \in H^{4}_{0}(0, L), v \in H^{2}_{0} (0,L), \nonumber \\ & & \zeta_{1}, 
%\zeta_{2} \in \mathbb{R}^{n}, \xi = J v_x(L), \psi = M v(L) \} \nonumber 
%\end{eqnarray} 
defined by 
\[ 
\mathcal{A} \left[ \begin{array}{c} u 
\\ v \\ \zeta_{1} \\ \zeta_{2} \\ \xi \\ \psi \end{array} \right] =
\left[ \begin{array}{c} v \\ - \frac{1}{\mu} (\Lambda u_{xx})_{xx} \\ A_{1} 
\zeta_{1} + b_{1}  \frac{\xi}{J} \\ A_{2} \zeta_{2} + b_{2}  \frac{\psi}{M}
 \\ - \Lambda(L) u_{xx}(L) - k_{1} u_x (L) - c_{1} \cdot \zeta_{1} - d_{1} 
\frac{\xi}{J} \\  (\Lambda u_{xx})_x(L) - k_{2}  u (L) - c_{2} \cdot 
\zeta_{2} - d_{2}  \frac{\psi}{M} \end{array} \right]. 
\]
%The reason for 
%introducing operator $\mathcal{A}$ is that for further analysis it is 
%beneficial to use semigroup theory to obtain results on solution existence and stability. 
Now we can write our problem as a first order evolution equation:  
\begin{equation} \label{evolution_notation}
\begin{array}{lcl}
z_t & = & \mathcal{A}z,\\
z(0) & = & z_{0} \in \mathcal{H}.
\end{array}
\end{equation}
For a review of abstract boundary feedback systems in a semigroup formalism we refer to \cite{KMN2003}.
The following well-posedness and stability result was obtained in \cite{Kugi:Thull}, for the homogeneous beam
(i.e. for $\mu$ and $\Lambda$ constant). The proof in the inhomogeneous case is performed analogously. 
Note that the contractivity of the semigroup also implies that $\|\cdot\|_{\mathcal{H}}$ 
is a Lyapunov functional for \eqref{evolution_notation}.
\begin{theorem} \label{well_posedness}
The operator $\mathcal{A}$ generates a $C_0$-semigroup of contractions 
on $\mathcal{H}$. For any $z_0 \in \mathcal{H}$, (\ref{evolution_notation}) has a unique mild solution $z\in 
C([0, \infty ) ; \mathcal{H})$  and $z(t)\stackrel{t\to\infty}{\longrightarrow}0$ in $\mathcal{H}$.
%it asymptotically converges to the zero state.  
\end{theorem}

But it remained an open question if this system is also exponentially stable. As a criterion we will use 
the following theorem due to Huang \cite{Huang}, which was also used for controlled EBBs 
\emph{without} tip mass \cite{Chen:Krantz:Ma:Wayne:West,Morgul}:
\begin{theorem}
 Let $T(t)$ be a uniformly bounded $C_0$-semigroup on a Hilbert
 space with infinitesimal generator $\mathcal{A}$. Then $T(t)$ is exponentially
 stable if and only if 
\begin{equation}
\label{real_part}
\sup{ \{ \emph{Re}(\lambda) \colon \lambda \in 
\sigma(\mathcal{A}) \} }  <  0
\end{equation}
and
\begin{equation}
\label{resolvent_norm}
\sup_{\lambda \in \mathbb{R}}{\| R( i \lambda, \mathcal{A})  \| }  < 
\infty 
\end{equation}
holds.
\end{theorem}

The following theorem is the main result of this section. Our proof of non-exponential stability of system  \eqref{evolution_notation} relies on the asymptotic behavior of its eigenvalues. A related spectral analysis of the inhomogeneous EBB, but with a boundary control torque is given in \cite{Guo2002}. Below we extend this study to the case when a dynamic control law is applied.

%There exists a number of papers concerned with the spectral analysis of the 
%Euler-Bernoulli beam, without and with tip mass, including the case 
%in which the beam is non-homogeneous (see \cite{Guo2002_II}, \cite{ChW2006}, 
%and \cite{ChW2007}). However, for the hybrid system of elasticity 
%(i.e. Euler-Bernoulli beam with a rigid body at the free end), the literature is not 
%so extensive. In \cite{Lit-Mar} and \cite{Guo-Wan} the Riesz basis property has been shown
%for the case of homogeneous EBB. Up to the knowledge of the authors, non-homogeneous case has 
%been considered solely in \cite{Guo2002}, where simple control torque is applied at the boundary. 
%In this section, we extend the study of the spectrum on the case when dynamic control law is applied.

\begin{theorem} \label{eigenvalue_asymtotic}
The operator $\mathcal{A}$ has eigenvalue pairs $\lambda_n$ and 
$\overline{\lambda}_n, \,n \in \mathbb{N}$, with the 
 following asymptotic behavior:

\begin{eqnarray*} 
\lambda_n &  = &  i \left[ \left( \frac{(2n - 1) \pi}{2 h} \right)^2 + 
\frac{4 h M^{-1}  \mu(L)^{\frac{3}{4}} \Lambda(L)^{\frac{1}{4}} - I}{2 h^2} \right] + \mathcal{O}(n^{-1}),
\end{eqnarray*} 
where 
\begin{equation} \label{h}
 h := \int_{0}^{L}{\left(\frac{\mu(w)}{\Lambda(w)}\right)^{\frac14} \;dw},
\end{equation} 
and $I$ is a real constant depending only on $\Lambda$, $\mu$, and given by \eqref{I}.
Therefore,  $$\sup{ \{ \emph{Re}(\lambda) \colon \lambda \in \sigma(\mathcal{A}) \} } = 0,$$ 
and hence the evolution problem (\ref{evolution_notation}) is \emph{not} 
exponentially stable.
\end{theorem}
\begin{proof}
We already know that the operator $\mathcal{A}$ has a compact resolvent
(see \cite{Kugi:Thull}). Thus, its spectrum $\sigma(\mathcal{A})$ consists
 entirely of isolated eigenvalues, at most countably many, 
and each eigenvalue has a finite algebraic multiplicity. 
Since $\mathcal{A}$ also  
generates an asymptotically stable $C_0$-semigroup of contractions we 
obtain 
\[ \quad \textnormal{Re}\lambda < 0, \quad \forall \lambda \in 
\sigma(\mathcal{A}).\] 
The matrices $A_1$ and 
$A_2$ are Hurwitz matrices and therefore only have eigenvalues with negative
 real parts. The set  $\sigma(\mathcal{A}) \cap (\sigma(A_1) \cup 
\sigma(A_2)) \subset \mathbb{C}$ is therefore empty or finite. Now we consider only such eigenvalues
$\lambda$ of the 
operator $\mathcal{A}$ that are not eigenvalues of $A_1$ or $A_2$. Then $z = (u, v, 
\zeta_{1}, \zeta_{2}, \xi, \psi)^{\top}\in D(\mathcal{A})$ is a corresponding eigenvector if and only if:
\begin{eqnarray*}
v& = & \lambda u, \\ 
\zeta_1 & = & -\lambda u_{x} (L) \left(A_1 - \lambda I \right)^{-1} b_{1}, \\
\zeta_2 & = & -\lambda u (L) \left(A_2 - \lambda I \right)^{-1} b_{2} ,
\end{eqnarray*} and 
\begin{eqnarray}
\left(\Lambda u_{xx}\right)_{xx} + \mu \lambda^2 u & = & 0, \label{e1} \\
u(0) & = & 0, \label{e2} \\ u_x (0) & = & 0, \label{e3} \\ 
 \Lambda(L) u_{xx} (L) + (k_1  - \lambda [ \left(A_1 - \lambda I
 \right)^{-1} b_{1} ] \cdot c_1  + \lambda d_1 + \lambda^2 J ) u_x(L)  & =& 0, \label{e4} 
\\-\left(\Lambda u_{xx}\right)_x(L) + (k_2  - \lambda  [ \left(A_2 - \lambda I
 \right)^{-1} b_{2} ]\cdot c_2 + \lambda d_2  + \lambda^2 M) u(L) & =& 0. \label{e5}   
\end{eqnarray} 
In order to solve \eqref{e1}--\eqref{e5}, we perform spatial transformations as in \cite{Guo2002_II}, which 
convert \eqref{e1} into a more convenient form.
{}First, \eqref{e1} is rewritten as:
\begin{equation} \label{ev_standard}
 u_{xxxx} + \frac{2 \Lambda_x}{\Lambda} u_{xxx} + \frac{ \Lambda_{xx}}{\Lambda} u_{xx} 
+ \frac{\mu}{\Lambda} \lambda^2  u = 0.
\end{equation}
Then a space transformation is introduced, so that the coefficient appearing with $u$ in \eqref{ev_standard} 
becomes constant. Let $u(x) = \breve u(y)$, where 
\begin{equation}\label{y}
 y =y(x) := \frac{1}{h} \int_{0}^{x}{\left(\frac{\mu(w)}{\Lambda(w)}\right)^{\frac14} \;dw},
\end{equation}
with $h$ defined as in \eqref{h}.
Then, from \eqref{e2}--\eqref{ev_standard} it follows that $\breve{u}$ satisfies:
\begin{equation} \label{e6} 
\begin{array}{rcc}
\breve{u}_{yyyy} + \alpha_3 \breve{u}_{yyy} + \alpha_2 \breve{u}_{yy} + \alpha_1 \breve{u}_{y} 
+ h^4 \lambda^2 \breve{u}\ & = & 0,\\
\breve{u}(0) & = & 0,  \\ \breve{u}_y (0) & = & 0, \\ 
 \breve{u}_{yy}(1) + \breve{u}_{y}(1)\left( \beta_0 + \kappa_1(\lambda) \right) & =& 0,
\\- \breve{u}_{yyy}(1) + \beta_1 \breve{u}_{yy}(1)  + \beta_2 \breve{u}_{y}(1) 
+ \kappa_2(\lambda) \breve{u}(1) & =& 0,   
\end{array}
\end{equation} 
with 
\begin{equation}\label{alpha_3}
\begin{array}{lcl}
 \alpha_3(y) &=& h \left(\frac{\mu(x)}{\Lambda(x)}\right)^{-\frac14} \left( \frac32 \frac{\mu_x(x)}{\mu(x)} 
 + \frac12 \frac{\Lambda_x(x)}{\Lambda(x)} \right),
\end{array}
\end{equation}
\begin{equation}\label{alpha_2}
\begin{array}{lcl}
 \alpha_2(y) & = & \frac{1}{h^2} \left\{ - \frac{9}{16} \left(\frac{\mu(x)}{\Lambda(x)}\right)^{-\frac32} 
\left[ \left( \frac{\mu(x)}{\Lambda(x)} \right)_x \right]^2 
+ \left(\frac{\mu(x)}{\Lambda(x)}\right)^{-\frac12} \left(\frac{\mu(x)}{\Lambda(x)}\right)_{xx} \right. \\
& + & \left. \frac{3}{2} \frac{\Lambda_x(x)}{\Lambda(x)} \left(\frac{\mu(x)}{\Lambda(x)}\right)^{-\frac12} 
\left(\frac{\mu(x)}{\Lambda(x)}\right)_{x} + \frac{\Lambda_{xx}(x)}{\Lambda(x)} \left(\frac{\mu(x)}{\Lambda(x)}\right)^{\frac12} \right\} ,
\end{array}
\end{equation}
and $\alpha_1$ is a smooth function of
$h$, $\frac{d^k \Lambda}{dx^{k}}$, and $\frac{d^k \mu}{dx^{k}}$ for $k=0,1,2,3$. 
The coefficients $\beta_0, \beta_1, \beta_2$ are 
constants, depending on $h$, $\frac{d^k \Lambda}{dx^{k}}(L)$, and $\frac{d^k \mu}{dx^{k}}(L)$ for $k=0,1,2$.
{}Furthermore, we have introduced the following notation: 
\begin{equation*}
\begin{array}{rcl}
\kappa_1 (\lambda) & := & \frac{h}{\Lambda(L)} \left( \frac{\mu(L)}{\Lambda(L)}\right)^{-\frac14}
\left( k_1 - \lambda \left( \left( A_1 - \lambda I \right)^{-1} b_1 \right) \cdot c_1 + \lambda d_1 
 + \lambda^2 J \right),\\
\kappa_2 (\lambda) & := & \frac{h^3}{\Lambda(L)} \left( \frac{\mu(L)}{\Lambda(L)}\right)^{-\frac34}
\left( k_2 - \lambda \left( \left( A_2 - \lambda I \right)^{-1} b_2 \right) \cdot c_2 + \lambda d_2
 + \lambda^2 M \right).
\end{array}
\end{equation*}
In order to solve \eqref{e6}, we use the strategy as in Chapter 2, Section 4 of \cite{Naimark}. 
Hence, to eliminate the third derivative term $\alpha_3 \breve{u}_{yyy}$, a new invertible space 
transformation is introduced: 
\begin{equation*}
\breve{u}(y) = e^{- \frac14 \int_0^y{\alpha_3(z) \,dz}} \tilde{u}(y).
\end{equation*}
Then \eqref{e6} becomes:
\begin{eqnarray}
\tilde{u}_{yyyy}  + \tilde{\alpha}_2 \tilde{u}_{yy} + \tilde{\alpha}_1 \tilde{u}_{y} 
+ \tilde{\alpha}_0 \tilde{u}+ h^4 \lambda^2 \tilde{u} & = & 0, \label{e11} \\
\tilde{u}(0) & = & 0, \label{e12} \\ \tilde{u}_y (0) & = & 0, \label{e13} \\ 
 \tilde{u}_{yy}(1) + \tilde{u}_{y}(1)\left( \beta_3 + \kappa_1(\lambda) \right) +
\tilde{u}(1)\left( \beta_4 - \frac14 \alpha_3(1) \kappa_1(\lambda) \right) & =& 0, \label{e14} 
\\- \tilde{u}_{yyy}(1) + \beta_5 \tilde{u}_{yy}(1)  + \beta_6 \tilde{u}_{y}(1) 
+ \left(\beta_7 + \kappa_2(\lambda)\right) \tilde{u}(1) & =& 0, \label{e15}   
\end{eqnarray} 
where
\begin{equation}\label{tilde_alpha_2}
 \tilde{\alpha}_2(y) = \alpha_2(y) - \frac38 \alpha_3(y)^2 - \frac32 (\alpha_3)_y(y),
\end{equation}
and $\tilde{\alpha}_1$, $\tilde{\alpha}_0$ are smooth functions of 
$h$, $\frac{d^k \Lambda}{dx^{k}}$, and $\frac{d^k \mu}{dx^{k}}$ for $k=0,\dots, 4$.
The constant coefficients $\beta_3, \dots, \beta_7$ depend on $h$, 
$\frac{d^k \Lambda}{dx^{k}}(L)$, and $\frac{d^k \mu}{dx^{k}}(L)$ for $k=0, \dots, 3$.
Due to the invertibility of the above transformations, the obtained problem \eqref{e11}--\eqref{e15} 
is equivalent to the original problem \eqref{e1}--\eqref{e5}. 

Since the eigenvalues of $\mathcal{A}$ come in complex conjugated pairs, and have negative real parts,
it suffices to consider only those $\lambda$ in the upper-left quarter-plane, i.e. such that 
$\arg{\lambda} \in (\frac{\pi}{2}, \pi]$. We define the unique $\tau \in \mathbb{C}$ such that 
$\text{Re}(\tau) \ge 0$, and $\lambda = i \frac{\tau^2}{h^2}.$ It can be seen that
$\arg{\tau} \in  (0, \frac{\pi}{4}]$. Now, the solution to \eqref{e11} can be approximated 
by the solution to the differential equation with the dominant terms only, i.e. 
$\tilde{u}_{xxxx} + \lambda^2 h^4 \tilde{u} = 0$. More precisely, we have 
(by adaptation of \emph{Satz 1}, pp. 42 of \cite{Naimark}; and the last result of Lemma \ref{ODE_solutions} 
is stated in the proof of \emph{Satz 1}):
\begin{lemma} \label{ODE_solutions}
For $\tau \in (0, \frac{\pi}{4}]$, and $|\tau|$ large enough, there exist linearly independent solutions 
$\{\gamma_j\}_{j=1}^{4}$, to \eqref{e11}, such that:
\begin{equation} \label{fundamental_asy}
\begin{array}{lcl}
\gamma_j(y) & = &  e^{\omega_j \tau y} \left( 1 + f_j(y) \right), \\
\frac{d^k}{dy^k} \gamma_{j}(y) &= &(\omega_j \tau)^k e^{\omega_j \tau y} 
\left( 1 + f_j(y) +\mathcal{O}(|\tau|^{-2}) \right), 
\quad k\in \{1, 2, 3\}, 
\end{array}
\end{equation}
where  $\omega_1 = 1, \, \omega_2 = i, \, \omega_3 = -1, \, \omega_4= -i$, and 
$$f_j(y) = - \frac{\int_0^y{\tilde{\alpha}_2(w) \, dw}}{4 \, \omega_j \tau} + \mathcal{O}(|\tau|^{-2}), 
\text{ as } |\tau| \rightarrow \infty, \,j = 1, \dots, 4.$$
Furthermore, the functions 
$\frac{d^k}{dy^k} \gamma_{j}$ depend analytically on $\tau$,
for $j= 1, \dots, 4, \, k = 0, \dots, 3$, and $|\tau|$ large enough. 
\end{lemma}
Now, due to Lemma \ref{ODE_solutions}, the solution to \eqref{e11}--\eqref{e15}
can be written as:
\[ \tilde{u}(y) = C_1 \gamma_1(y) + C_2 \gamma_2(y) +
  C_3 \gamma_3(y) + C_4 \gamma_4(y), \]
where the constants $\{C_j\}_{j=1}^4$ are determined by the boundary conditions 
\eqref{e12} -- \eqref{e15}, and therefore satisfy the following linear system:
\begin{equation} \label{eig_system}
\begin{array}{ccl}
0 & = & C_1 \gamma_1(0) + C_2 \gamma_2(0) + C_3 \gamma_3(0) + C_4 \gamma_4(0),\\ 
0 & = & C_1 (\gamma_1)_y(0) + C_2 (\gamma_2)_y(0) + C_3 (\gamma_3)_y(0) + C_4 (\gamma_4)_y(0),\\
0 & = & \sum_{i=1}^{4}{C_i m_{3\,i}},\\
0 & = & \sum_{i=1}^{4}{C_i m_{4\,i}},
\end{array}
\end{equation}
where we define:
\[m_{3\,i} :=(\gamma_i)_{yy}(1) + 
(\beta_3 + \kappa_1(\lambda))(\gamma_i)_{y}(1) + (\beta_4 - \frac{1}{4} \alpha_3(1) \kappa_1(\lambda))
 \gamma_i(1),\]
\[m_{4\,i} :=  -(\gamma_i)_{yyy}(1) + \beta_5 (\gamma_i)_{yy}(1) + 
\beta_6 (\gamma_i)_{y}(1) + (\beta_7 + \kappa_2(\lambda))
 \gamma_i(1).\]
{}From \eqref{fundamental_asy} easily follows: 
\begin{equation} \label{asy_behaviour}
\begin{array}{l}
\gamma_j(0) = 1 + f_j(0), 
\quad (\gamma_j)_y(0) = \omega_j \tau (1 + f_j(0) + \mathcal{O}(|\tau|^{-2})), \quad j=1, \dots, 4, \\
m_{31} = e^{\tau} \left( (l_1 \tau^5 + l_2 \tau^4)(1+f_1(1)) + \mathcal{O}(|\tau|^{3})\right), \\
m_{41} = e^{\tau} \left( (l_3 \tau^4 - \tau^3)(1+f_1(1)) + \mathcal{O}(|\tau|^{3})\right), \\
m_{32} = e^{i \tau}  \left( (i l_1 \tau^5 + l_2 \tau^4)(1+f_2(1)) + \mathcal{O}(|\tau|^{3})\right), \\
m_{42} = e^{i \tau}  \left( (l_3 \tau^4 + i \tau^3)(1+f_2(1)) + \mathcal{O}(|\tau|^{2})\right), \\
m_{33} = e^{- \tau}  \left( (-l_1 \tau^5 + l_2 \tau^4)(1+f_3(1)) + \mathcal{O}(|\tau|^{3})\right), \\
m_{43} = e^{- \tau}  \left( (l_3 \tau^4 +  \tau^3)(1+f_3(1)) + \mathcal{O}(|\tau|^{2})\right)), \\
m_{34} = e^{-i \tau}  \left( (-i l_1 \tau^5 + l_2 \tau^4)(1+f_4(1)) + \mathcal{O}(|\tau|^{3})\right), \\
m_{44} = e^{-i \tau}  \left( (l_3 \tau^4 - i \tau^3)(1+f_4(1)) + \mathcal{O}(|\tau|^{2})\right), 
\end{array}
\end{equation}
with \[l_1:= -\frac{J}{h^3 \Lambda(L)} \left(\frac{\mu(L)}{\Lambda(L)}\right)^{-\frac14}, 
l_2:= \frac{J \alpha_3(1)}{4 h^3 \Lambda(L)} \left(\frac{\mu(L)}{\Lambda(L)}\right)^{-\frac14}, 
l_3:= -\frac{M}{h \Lambda(L)} \left(\frac{\mu(L)}{\Lambda(L)}\right)^{-\frac34}.\] 
{}For $\tilde{u}$ to be nontrivial, the determinant of the system \eqref{eig_system} has to vanish:
\begin{equation} \label{det}
\begin{vmatrix}
\gamma_1(0) & \gamma_2(0) & \gamma_3(0) & \gamma_4(0)\\ 
(\gamma_1)_y(0) & (\gamma_2)_y(0) & (\gamma_3)_y(0) & (\gamma_4)_y(0)\\
m_{31} & m_{32} & m_{33} & m_{34} \\
m_{41} & m_{42} & m_{43} & m_{44} 
\end{vmatrix} = 0.
\end{equation}
%Note that each solution $\tau \in \mathbb{C}$ (with $\textnormal{Re} 
%(\tau) > 0$) of (\ref{cond}) yields a nontrivial $u_{\tau}$ and hence an
% eigenvector
% $z_{\tau}$ of $\mathcal{A}$, corresponding to $\lambda = i \tau^2 \gamma$.
Next we shall write (\ref{det}) in an asymptotic
form when $\textnormal{Re} (\tau)$ is large: %To this end note that the dominant terms of \eqref{asy_behaviour} 
%are $m_{31}, m_{41}, m_{34}$ and $m_{44}$ since $\tau \in (0, \frac{\pi}{4}]$.
\begin{equation} \label{det_asy}
 B_1 (m_{31} m_{44} - m_{41} m_{34}) %+ i B_2 (m_{31} m_{43} - m_{41} m_{33}) 
+ B_2 (m_{31} m_{42} - m_{41} m_{32}) + \mathcal{O}(|\tau|^{10}) = 0,
\end{equation}
where 
\begin{equation} \label{Bs}
\begin{array}{ccl}
B_1 & := & - (1+i)\left[1 + f_2(1) + f_3(1)\right] + \mathcal{O}(|\tau|^{-2}), \\ 
%B_2 & := & 2\left[1+ f_2(1) + f_4(1)\right] + \mathcal{O}(|\tau|^{-2}), \\ 
B_2 & := &  (1-i)\left[1+ f_3(1) + f_4(1)\right] + \mathcal{O}(|\tau|^{-2}). 
\end{array}
\end{equation}
Noting only the terms with leading powers of $\tau$ in 
(\ref{det_asy}), and after division by $e^{\tau} \tau^{10}$, we obtain
 \begin{eqnarray}\label{equa}
\cos{\tau} - \tau^{-1} (  \frac{I}{4} + \frac{1}{l_3}) (\cos{\tau} + \sin{\tau}) 
+ \mathcal{O}(|\tau|^{-2})& = & 0,
\end{eqnarray} 
where 
\begin{equation} \label{I}
 I := \int_{0}^{1}{\tilde{\alpha}_2(w) \, dw}.
\end{equation}
We set $k = n - \frac{1}{2}$ for $n \in \mathbb{N}$ sufficiently large and 
consider equation (\ref{equa}) for $\tau$ in a neighborhood of $k \pi$. 
We shall apply Rouch\'{e}'s Theorem (see \cite{Krantz}, e.g.) to the equation (\ref{equa}),
written as 
\begin{equation} \label{rouche}
 \cos{\tau} + f(\tau)  = 0,
\end{equation} 
where $f(\tau) = \mathcal{O} (|\tau|^{-1})$.  
Consider $\cos{\tau}$ on a simple closed contour $K \subset \{ (n-1)
 \pi \le  \textnormal{Re} (\tau) \le n \pi\}$ ``around'' 
$\tau = k \pi$ such that $| \cos{\tau} | \ge 1$ on $K$. 
For $n$ large enough, the holomorphic function $f$ satisfies $|f(z)| < 1 \le 
|\cos{\tau}|$ on $K$. Since $\tau = k \pi$ is the only zero of 
$\cos{\tau}$ inside $K$, Rouch\'{e}'s Theorem implies that (\ref{rouche})
has also exactly one solution inside $K$:
\begin{equation} \label{tau_n} 
\tau_n = k \pi + h_n. 
\end{equation}  
Then, $\cos{\tau_n} = (-1)^{n} 
\sin{h_n}$. Furthermore, (\ref{rouche}) implies $h_n = \mathcal{O}(n^{-1})$. 
To make the asymptotic behavior of $h_n$ more precise, we consider
\begin{eqnarray*}
\sin{\tau_n} & = &  -(-1)^n \cos{h_n} =  -(-1)^n + \mathcal{O}(n^{-2}) ,\\
\cos{\tau_n} & = & (-1)^n \,h_n  + \mathcal{O}(n^{-3}).
\end{eqnarray*}
Using this in (\ref{equa}) we get \[ h_n + \tau^{-1} (\frac{1}{l_3} + \frac{I}{4}) + \mathcal{O}(n^{-2}) = 0.\]
{}Finally, this yields \[ h_n = \frac{4 h M^{-1} \mu(L)^{\frac{3}{4}} \Lambda(L)^{\frac{1}{4}} - I}{4 k \pi }
 +  \mathcal{O}(n^{-2}),\]
 and \eqref{tau_n} implies
\begin{equation} \label{asy_eigenvalue}
%\begin{aligned}
\lambda_n  =  i \left(\frac{\tau_n}{h}\right)^2   =
 i \left[ \left( \frac{k \pi}{h} \right)^2 + 
\frac{4 h M^{-1}\mu(L)^{\frac{3}{4}} \Lambda(L)^{\frac{1}{4}} - I}{2 h^2} \right] + \mathcal{O}(n^{-1}).
%\end{aligned}
\end{equation}
Hence, condition (\ref{real_part}) fails and $T(t)$ is \emph{not} exponentially stable.
\end{proof}
In Figure \ref{Eigenvalue_plot} we show the eigenvalue pairs corresponding to the simulation example from \S\ref{S-simul}.
 They were obtained by application of Newton's method to the equation (\ref{det}).
\begin{figure}[h]
\begin{center}
 \includegraphics[trim = 0mm -100mm 0mm 0mm,scale=0.5]{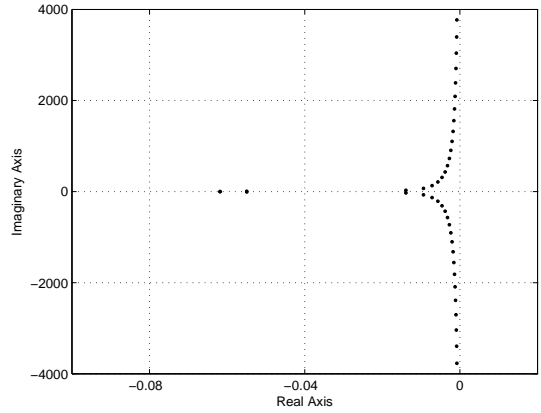}
\end{center}
\vspace{-5cm}
\caption{The eigenvalues $\lambda_n$ of the system approach the imaginary axis as $n \rightarrow \infty$.}
\label{Eigenvalue_plot}
\end{figure}
\begin{remark}
It can also be shown that the condition (\ref{resolvent_norm})
 does not hold. In particular, it can be shown that there is a constant
 $C$,  a sequence $\{ \mu_n \} \subset \mathbb{R}$ diverging to $ + \infty$,
and a sequence $\{ z_n \} \subset D( \mathcal{A})$ such that \[ \frac{ \| R(i
 \mu_n, \mathcal{A}) z_n \|_{\mathcal{H}} }{\| z_n \|_{\mathcal{H}}} > C \mu_n,
 \quad \text{ for all } n \text{ large enough.}\] 
But since
the details of this calculation are rather technical we only present them 
in \cite{Miletic}. 
\end{remark}
\begin{remark}
We shall now comment on the asymptotic behavior of the eigenfunctions of $\mathcal{A}$. 
The solution to \eqref{e11}--\eqref{e15} for $\tau = \tau_n$ has the form (see \cite{Naimark}):
\[ \tilde{u}_n(y) = 
\begin{vmatrix}
\gamma_1(0) & \gamma_2(0) & \gamma_3(0) & \gamma_4(0)\\ 
(\gamma_1)_y(0) & (\gamma_2)_y(0) & (\gamma_3)_y(0) & (\gamma_4)_y(0)\\
m_{31} & m_{32} & m_{33} & m_{34} \\
\gamma_1(y) & \gamma_2(y) & \gamma_3(y) & \gamma_4(y)
\end{vmatrix},
\]  
up to a multiplicative constant. Using the Laplace expansion of the determinant and scaling the expression with 
$e^{-\tau} \tau^{-6}\frac{1}{l_1 2i}$, $\tilde{u}_n$ has the approximate form (for $n$ large):
$$
\tilde{u}_n(y) =  e^{ - (n - \frac12) \pi y} - \cos {\left( (n - \frac12) \pi y\right)}  + 
\sin {\left( (n - \frac12) \pi y\right)} + (-1)^{n} e^{ (n - \frac12) \pi (y-1)} 
+\mathcal{O}(n^{-1}),
$$ for $0 \le y \le 1$. Therefore, the function $u_n$ corresponding to the eigenvalue 
$\lambda_n$ has the following asymptotic property:
\begin{eqnarray*}
u_n(x) &=&  e^{- \frac14 \int_0^y{\alpha_3(z) \, dz}} \left[ e^{ - (n - \frac12) \pi y} - \cos {\left( (n - \frac12) \pi y\right)}  + 
\sin {\left( (n - \frac12) \pi y\right)} \right. \\ & & \left. + (-1)^{n} e^{ (n - \frac12) \pi (y-1)} 
+\mathcal{O}(n^{-1}) \vphantom{\cos {\left( (n - \frac12) \pi y\right)}} \right],
\end{eqnarray*}
where $0 \le x \le L$, with $y= y(x)$ and $\alpha_3$ as in \eqref{y} and \eqref{alpha_3}.
\end{remark}
\begin{remark}
The uncontrolled system (i.e. with $A_{1,2} = 0, d_{1,2} = 0$) is 
undamped and its operator $\mathcal{A}$ then has purely imaginary eigenvalues. But 
their asymptotic behavior is still like in Theorem \ref{eigenvalue_asymtotic}, 
as can be verified by the analogue of the above computation.
\end{remark}

%%%%%%%%%%%%%%%%%%%%%%%%%%%%%%%%%%%%%%%%%%%%%%%%%%%%%%%%%%%%%%%%%%%%%%%%%%%%%%%
\section{Dissipative FEM method}\label{S-FEM}

{}From Theorem \ref{well_posedness} we know that the norm of the solution
$z(t)$ decreases in time. Using (\ref{kyp}), a straightforward calculation (for a classical solution) yields:
\begin{eqnarray}\label{norm-decay}
\frac{d}{dt} \| z \|^2_{\mathcal{H}} & = &  - \delta_1 u_{xt}(L)^2
-\frac{1}{2} \left(\zeta_{1} \cdot q_1 + \tilde\delta_1 u_{xt}(L)\right)^2  \nonumber \\ 
&  & - \delta_2 u_{t}(L)^2-\frac{1}{2} \left(\zeta_{2} \cdot q_2 + \tilde\delta_2 u_{t}(L)\right)^2 \\ 
&  & - \frac{\epsilon_1}{2} \zeta_{1}^{\top} P_1 \zeta_1- 
\frac{\epsilon_2}{2} \zeta_{2}^{\top} P_2 \zeta_2  \le 0 ,\nonumber
\end{eqnarray} 
where $\tilde\delta_j = \sqrt{2(d_j - \delta_j)}, \; j = 1, 2.$
Note that the r.h.s.\ of \eqref{norm-decay} only involves boundary terms of the beam and the control variables. 
Hence, $\frac{d}{dt} \| z \|^2_{\mathcal{H}}=0$ does \emph{not} imply $z=0$ (which can easily be verified from \eqref{evolution_notation}).

The goal of this section is to derive a FEM for (\ref{model1})--(\ref{model5}) coupled to the ODE-system 
(\ref{control_law}) that preserves this structural property of dissipativity. 
The importance of this feature is twofold: For long-time computations, the numerical scheme must of course 
be convergent in the classical sense (i.e.\ on finite time intervals) but also yield the correct large-time 
limit. Moreover, dissipativity of the scheme implies immediately unconditional stability.  

Here we shall construct first a 
time-continuous and then a time-discrete FEM that both dissipate the norm in time. 
Let us briefly discuss the different options to proceed.
(\ref{evolution_notation}) is an inconvenient starting point for deriving a weak formulation due to the high 
boundary traces of $u$ at $x=L$: The natural regularity of a weak solution would be 
$u \in C([0, \infty) ; \tilde H^2_0 (0,L))$, $v = u_t \in C([0, \infty) ;L^2 (0,L))$. Hence, the terms 
$\Lambda(L) u_{xx}(t,L)$, $(\Lambda u_{xx})_x(t,L)$ in (\ref{evolution_notation}) could only be incorporated by resorting to 
the boundary conditions \eqref{model4}, \eqref{model5}. Therefore we shall rather start from the original 
second order system \eqref{model1}--\eqref{control_law}.

%%%%%%%%%%%%%%%%%%%%%%%%%%%%%%%%%%%%%%%%%%%%%%%%%%%%%%%%%%%%%%%%%%%%%%%%%%%%%%%%%
\subsection{Weak formulation}

In order to derive the weak formulation, we assume the following initial conditions
\begin{subequations} \label{init_cond}
  \begin{align}
u(0) & =  u_0 \in \tilde H_0^2(0, L), \\ u_t(0) & =  v_0 \in L^2(0, L), \label{model7} \\ 
\zeta_1 (0) & =  \zeta_{1,0} \in \mathbb{R}^n, \label{model8} \\ 
\zeta_2 (0) & = \zeta_{2,0} \in \mathbb{R}^n. \label{model9} 
\end{align}
\end{subequations}
Moreover, let $v_0(L)$ and $(v_0)_x(L)$ be given in addition 
to the function $v_0$, and \textit{not} as its trace.
Multiplying (\ref{model1}) by $ w \in \tilde{H}^2_0(0,L)$, 
integrating over $[0,L]$, and  taking into account the given boundary
conditions we obtain:
\begin{eqnarray}\label{motivation-form}
& & \int_{0}^{L}{\mu u_{tt} w \,dx} + \int_{0}^{L}{\Lambda u_{xx}
 w_{xx} \,dx}  + M u_{tt}(t, L)w(L) + J u_{ttx}(t, L)w_{x}(L) \nonumber \\ 
& & + k_1 u_x(t, L) w_x(L) + k_2 u(t, L) w(L) + d_1 u_{tx}(t, L) w_x(L) + d_2 u_{t}(t, L) w(L) \\ 
& &  + c_1 \cdot \zeta_1(t) \; w_x(L) + 
 c_2 \cdot \zeta_2(t) \; w(L)= 0, %\\ &&  
\quad\quad\quad\quad\quad\forall w \in \tilde H^2_0(0,L), \;  t > 0.\nonumber 
\end{eqnarray}
This identity will motivate the weak formulation. First, we define the Hilbert space 
$$H := \mathbb R \times \mathbb R \times L^2(0,L),$$ with inner product
\begin{eqnarray*}
 ( \hat{\varphi}, \hat{\nu})_H & := & J \,(^1\hat{\varphi}) \, (^1\hat{\nu}) +M \,(^2\hat{\varphi}) \,(^2\hat{\nu}) + 
(\mu \,^3\hat{\varphi},\,^3\hat{\nu})_{L^2}, 
\end{eqnarray*}
for $\hat{\varphi} = (^1\hat{\varphi}, \,^2\hat{\varphi}, \,^3\hat{\varphi}),\, \nu \in H.$ 
We also define the Hilbert space
$$V := \{ \hat{w} = (w_x(L), w(L), w) \colon w \in \tilde H^2_0(0,L)\},$$ with the inner product
$$(\hat{w_1},\hat{w_2})_V = (\Lambda (w_1)_{xx}, (w_2)_{xx})_{L^2}.$$
It can be shown that $V$ is densely embedded in $H$. Therefore taking $H$ as a pivot space, we have the Gelfand triple
$$V \subset H \subset V'.$$

{}For any fixed $T>0$ we now define  $\hat{u} = (u_{x}(L),  u(L), u)$ and $\zeta_1, \zeta_2$ to be the \textit{weak solution} 
to (\ref{model1})--(\ref{control_law}) and \eqref{init_cond} if 
$$ \hat{u} \in L^2(0, T; V) \cap H^1(0, T; H) \cap H^2(0, T; V'),$$ $$\zeta_1, \zeta_2 \in H^1(0,T;\mathbb{R}^n)$$
and it satisfies: 
\begin{equation}\label{weak-form}
_{V'}<\hat{u}_{tt}, \hat{w}>_{V} + a(\hat{u}, \hat{w}) + b(\hat{u}_t, \hat{w}) + 
e_1(\zeta_1, \hat{w}) + e_2(\zeta_2, \hat{w})= 0,
\end{equation} 
$\text{for a.e. }t \in (0, T), \forall \hat{w} \in V$. 
The bilinear form $_{V'}<. , .>_{V}$ is the duality pairing between $V$ and $V'$ as a natural extension of the inner product in $H$.
The bilinear forms $a : V \times V \rightarrow \mathbb{R}$, $b : H \times H \rightarrow \mathbb{R}$ and 
$e_1, e_2 : \mathbb{R}^n \times V \rightarrow \mathbb{R}$ are given by
\begin{eqnarray*}
a(\hat{w_1},\hat{w_2}) & = & (\hat{w_1},\hat{w_2})_V + 
k_1 (w_1)_x(L) (w_2)_x(L) + k_2w_1(L) w_2(L), \\
b(\hat{\varphi},\hat{\nu}) & = & d_1 (^1\hat{\varphi}) (^1\hat{\nu}) + d_2 (^2\hat{\varphi}) (^2\hat{\nu}), \\
e_1(\zeta_1, \hat{w}) &= & c_1 \cdot \zeta_1 w_x(L), \\
e_2(\zeta_2, \hat{w}) &= & c_2 \cdot \zeta_2 w(L).
\end{eqnarray*}
Equation (\ref{weak-form}) is coupled to the ODEs
\begin{eqnarray}\label{zeta-ODE}
\begin{array}{rcl}
(\zeta_{1})_t(t)  & = & A_{1} \zeta_{1}(t) + 
b_{1} \, (^1\hat{u}_{t}(t)), \\  
(\zeta_{2})_t(t)  & = & A_{2} \zeta_{2}(t) + b_{2} \, (^2\hat{u}_{t}(t)),
\end{array}
\end{eqnarray}
with initial conditions 
\begin{subequations} \label{hat_ic}
  \begin{align}
 \hat{u}(0)  &=  \hat{u}_0 = ((u_0)_x(L), u_0(L), u_0) \in V, \label{hat_ic_a}\\
 \hat{u}_t(0) &= \hat{v}_0 = ((v_0)_x(L), v_0(L), v_0) \in H, \label{hat_ic_b} \\
 \zeta_1(0)  &=  \zeta_{1,0} \in \mathbb{R}^n, \\
 \zeta_2(0)  &=  \zeta_{2,0} \in \mathbb{R}^n.  
\end{align}
\end{subequations}
In (\ref{hat_ic_a}) the first two components of the right hand side are the boundary traces of $u_0 \in \tilde H^2_0(0, L)$,
 but in (\ref{hat_ic_b}) they are additionally given values. 
Note that in the case when
 $\hat{u} \in H^2(0, T; V)$, formulation (\ref{weak-form}) is equivalent to identity (\ref{motivation-form}). 
This weak formulation is an extension of \cite{Banks:Rosen}(Section 2) to the case where 
the beam with the tip-mass is additionally coupled to the first order ODE controller system. Here, we have to deal also with $u_t(L)$ and $u_{tx}(L)$.
And these additional first order boundary terms (in $t$), included in $b(.,.)$, require a slight generalization of the 
standard theory (as presented in \S8 of \cite{Lions:Magenes}, e.g.).

In order to give a meaning to the initial conditions (\ref{hat_ic_a}), (\ref{hat_ic_b}) we shall use the following lemma 
(special case of Theorem 3.1 in \cite{Lions:Magenes}).
\begin{lemma} \label{Continuity_Lions}
Let $X$ and $Y$ be two Hilbert spaces, such that $X$ is dense and continuously embedded in $Y$. Assume that 
\begin{eqnarray*}
u & \in & L^2(0, T; X), \\ u_t & \in & L^2(0, T; Y).
\end{eqnarray*} 
Then $$u \in C([0, T]; [X, Y]_{\frac{1}{2}}]),$$  
after, possibly, a modification on a set of measure zero. Here, the definition of \emph{intermediate spaces} as given in 
\cite{Lions:Magenes}, \S2.1, was assumed.
\end{lemma}
\begin{theorem} \label{existence_weak_sol}
\begin{enumerate}[(a)]
\item The weak formulation (\ref{weak-form}) -- (\ref{hat_ic}) has a unique solution $(\hat{u}, \zeta_1, \zeta_2)$.
\item The weak solution has the additional regularity
\begin{subequations} 
  \begin{align}
 & \hat{u} \in L^{\infty}(0,T;V), \quad \hat{u}_t \in L^{\infty}(0,T;H), & \label{reg_infty}\\
& \zeta_1, \zeta_2 \in C([0,T]; \mathbb{R}^n), & \label{reg_zeta}\\
& \hat{u} \in C([0,T];[V,H]_{\frac12}), & \label{reg_con_u}\\ & \hat{u}_t \in C([0,T];[V,H]_{\frac12}^{'}). & \label{reg_con_u_t} 
\end{align}
\end{subequations}
\end{enumerate}
\end{theorem} 
Furthermore, even stronger continuity for the weak solution can be shown:
\begin{theorem} \label{continuity_weak_sol}
 After, possibly, a modification on a set of measure zero, the weak solution $\hat{u}$ of (\ref{weak-form})-(\ref{hat_ic}) satisfies
\begin{eqnarray*}
\hat{u}   & \in & C([0,T]; V), \\
\hat{u}_t & \in & C([0,T]; H).
\end{eqnarray*}
\end{theorem}
The proofs of Theorem \ref{existence_weak_sol} and \ref{continuity_weak_sol} are given 
in Appendix A.

%%%%%%%%%%%%%%%%%%%%%%%%%%%%%%%%%%%%%%%%%%%%%%%%%%%%%%%%%%%%%%%%%%%%%%%%%%%%%%%%%
\subsection{Semi-discrete scheme: space discretization}

Now let $W_h \subset \tilde H_0^2(0,L)$ be a finite dimensional space. Its elements
are globally $C^1[0,L]$, due to a Sobolev embedding.
{}For some fixed basis $w_j, j=1, \dots, N$ the Galerkin approximation of (\ref{weak-form}) reads:
{}Find $u_h \in C^2([0, \infty), W_h)$, i.e. $\hat{u}_h= ((u_h)_x(L), u_h(L), u_h) \in C^2([0, \infty), V)$, and 
$\tilde \zeta_{1,2} \in C^1([0, \infty),\mathbb{R}^n)$ with
\begin{equation} \label{fem} 
\begin{array}{ccl}
 && \int_{0}^{L}{\mu (u_{h})_{tt} w_j \,dx}   + 
\int_{0}^{L}{\Lambda (u_h)_{xx} (w_j)_{xx} \,dx}   \\
 & & + M (u_h)_{tt}(t,L)w_j(L) + J (u_h)_{xtt}(t,L)(w_j)_{x}(L) 
 \\ & & + k_1 (u_{h})_x(t,L) (w_{j})_x(L) +  k_2 u_h(t,L) w_j(L)  \\ 
& & + d_1 (u_h)_{xt}(t,L) (w_j)_x(L) + d_2 (u_h)_{t}(t,L) (w_j)(L)  \\ 
& & + c_1 \cdot \tilde\zeta_1(t) \; (w_j)_x(L) + 
 c_2 \cdot \tilde\zeta_2(t) \; w_j(L)= 0, %\\ & & 
\quad\quad\quad j=1, \dots, N, \; t > 0,
\end{array}
\end{equation}
coupled to the analogue of \eqref{zeta-ODE}:
\begin{eqnarray}\label{zeta-tilde-ODE}
\begin{array}{rcl}
(\tilde\zeta_{1})_t(t)  & = & A_{1} \tilde\zeta_{1}(t) + 
b_{1}  (u_h)_{xt} (t, L), \\  
(\tilde\zeta_{2})_t(t)  & = & A_{2} \tilde\zeta_{2}(t) + b_{2} (u_h)_{t} (t, L),
\end{array}
\end{eqnarray}
and the initial conditions
\begin{eqnarray*}
u_h(0, \ldotp) & = & u_{h,0} \in W_h, \\ (u_{h})_t(0, \ldotp) & = & v_{h,0} \in W_h,
 \\ \tilde\zeta_1(0) & = & \zeta_{1,0} \in \mathbb{R}^n, \\ \tilde\zeta_2(0) & = & \zeta_{2,0}
\in \mathbb{R}^n. 
\end{eqnarray*}
\eqref{fem}  is a second order ODE-system in time. 
Expanding its solution in the chosen basis, i.e.\
 \[u_h(t,x) = \sum_{i=1}^{N}{U_i(t) w_i(x)},\] 
and denoting its coefficients by the vector 
\[ \mathbb{U} = \left[ \begin{array}{c c c c} U_1 &
 U_2 & \dots & U_N \end{array} \right]^{\top} \] 
yields the equivalent vector equation: 
\begin{equation}\label{vector-system}
 \mathbb{A} \mathbb{U}_{tt} + \mathbb{B}
 \mathbb{U}_{t} + \mathbb{K} \mathbb{U} + \mathbb{C}(t) = 0 .
\end{equation}
Its coefficient matrices are defined as 
\begin{eqnarray*} 
\mathbb{A}_{i,j} &:=& \int_{0}^{L}{\mu \, w_i
 w_j \,dx} + M w_i(L)w_j(L) + J (w_i)_{x}(L)(w_j)_{x}(L),\\
\mathbb{B}_{i,j} &:=& d_1 (w_i)_{x}(L)(w_j)_{x}(L) +  d_2 w_i(L)w_j(L),\\
\mathbb{K}_{i,j} &:=&  \int_{0}^{L}{\Lambda (w_i)_{xx} (w_j)_{xx} \,dx}
+k_1 (w_i)_{x}(L)(w_j)_{x}(L) +  k_2 w_i(L)w_j(L),\\
&& i,j = 1, \dots, N,
\end{eqnarray*}
and the vector $\mathbb{C}$ has the entries
 \[ \mathbb{C}_j(t) =  c_1 \cdot \tilde\zeta_1(t) \; (w_j)_x(L) + 
 c_2 \cdot \tilde\zeta_2(t) \; w_j(L), \quad j = 1, \dots, N.\]
The matrix $\mathbb{K}$ is symmetric positive definite, since we assumed $k_{1,2}>0$.
Since also $\mathbb{A}$ is symmetric positive definite, one sees very easily that the IVP corresponding to
the coupled problem \eqref{vector-system}, \eqref{zeta-tilde-ODE} is uniquely solvable.

For a final specification of the FEM we need to choose an appropriate discrete space. 
Only for notational simplicity, we shall assume a uniform distribution of nodes on $[0,L]$:
\[ x_m = m h, \quad m \in \{0, 1, \dots, P\}, \] 
where $h=\frac{L}{P}.$
A standard choice for the discrete space $W_h$ is a space of piecewise cubic polynomials with both 
displacement and slope continuity across element boundaries, also called Hermitian cubic polynomials (see 
\cite{Shames:Dym}, 
\cite{Bar-Yoseph:Fisher:Gottlieb}, e.g.).
They have been employed not only for the Euler-Bernoulli beam, but also 
Timoshenko beams (cf.\ \cite{Falsone:Settineri}).
To define a basis for $W_h$ (Hermite cubic basis, see e.g. \cite{Scott}), 
we associate two piecewise cubic functions with each node $x_m,\;m\ge1$ satisfying:
\begin{eqnarray*}
w_{2m-1} (x_k) = \left\{ \begin{array}{c l}
1, & m=k \\ 0, & m \ne k
\end{array} \right. &\qquad & w'_{2m-1} (x_k) = 0 ,\\ 
w'_{2m} (x_k) = \left\{ \begin{array}{c l}
1, & m=k \\ 0, & m \ne k
\end{array} \right. & \qquad & w_{2m} (x_k) = 0,
\end{eqnarray*} 
for all $k = 0, \dots, P$. Hence, the nodal values of a function and of its derivative are 
the associated degrees of freedom. Due to the boundary conditions at $x=0$ in $W_h\subset \tilde H_0^2$, 
the basis set does not include the functions $w_{-1}$ and $w_{0}$ associated to the node $x_0=0$. Thus, $N=2P$.
For the coupling to the control variables we shall need the boundary values of $u_h$. The above basis 
yields the simple relations $u_h(t,L)=U_{N-1}(t),\;(u_h)_x(t,L)=U_{N}(t)$. Compact support of the 
basis functions $\{ w_j \}_{j=1}^{N}$ leads to a sparse structure of the matrices $\mathbb{A}$, $\mathbb{B}$, and $\mathbb{K}$:
 $\mathbb{A}$ and $\mathbb{K}$ are tridiagonal, $\mathbb{B}$ is diagonal with only two non-zero elements $\mathbb{B}_{N-1, N-1} = d_2$, 
$\mathbb{B}_{N, N} = d_1$. And the vector $\mathbb{C}$ has all zero entries except for $\mathbb{C}_{N-1} = c_2 \cdot \tilde{\zeta_2}$, 
$\mathbb{C}_{N} = c_1 \cdot \tilde{\zeta_1}$.

Next, we shall show that the semi-discrete solution $u_h(t)$ decreases in time. As an analogue of the norm 
$\|z(t)\|_{\mathcal{H}}$ from \S\ref{s2}, we first define the following time dependent functional for a 
trajectory $u \in  C^2([0, \infty) ; \tilde H^2_0 (0,L))$ and $\zeta_{1,2} \in C^1([0, \infty) ; \mathbb{R}^n)$:
\begin{eqnarray} \label{energy_functional}
  E(t; u,\zeta_{1},\zeta_{2}) \nonumber
  &\!\!:=\!\!&\frac12 \int_{0}^{L}\left(\Lambda u_{xx}(t,x)^2+\mu u_t(t,x)^2\right)\,dx
  +\frac{M}{2}u_t(t,L)^2+\frac{J}{2}u_{xt}(t,L)^2\\
  &&+\frac{k_1}{2}u_x(t,L)^2 + \frac{k_2}{2}u(t,L)^2
  +\frac12 \zeta_1^\top(t) P_1 \zeta_1(t)+\frac12 \zeta_2^\top(t) P_2 \zeta_2(t).
\end{eqnarray} 
For a classical solution of \eqref{evolution_notation} in $D(\mathcal A)$ we have
$E(t;u,\zeta_{1},\zeta_{2})=\|z(t)\|_{\mathcal{H}}^2$.

\begin{theorem}\label{th:semidiscrete-decay}
Let $u_h \in  C^2([0, \infty) ; \tilde H^2_0 (0,L))$ and $\tilde\zeta_{1,2} \in C^1([0, \infty) ; 
\mathbb{R}^n)$ solve \eqref{fem}, \eqref{zeta-tilde-ODE}. Then it holds for $t>0$:
\begin{eqnarray*}
\frac{d}{dt}E(t;u_h,\tilde\zeta_{1},\tilde\zeta_{2})& = &  - \frac{\epsilon_1}{2}
 \tilde\zeta_{1}^{\top} P_1 \tilde\zeta_1 -\frac{1}{2} \left(\tilde\zeta_{1} \cdot q_1 + 
\tilde\delta_1 (u_h)_{xt}(L)\right)^2 - \delta_1 (u_h)_{xt}(L)^2 \\
 &  & - \frac{\epsilon_2}{2} \tilde\zeta_{2}^{\top} P_2 \tilde\zeta_2 -
\frac{1}{2} \left(\tilde\zeta_{2} \cdot q_2 + \tilde\delta_2
 (u_h)_{t}(L)\right)^2 - \delta_2 (u_h)_{t}(L)^2 \le 0.
\end{eqnarray*} 
\end{theorem} 
\begin{proof}
In the following computation we use  (\ref{fem}) with the test function $w_h= (u_h)_t$: 
\begin{eqnarray*}
\frac{d}{dt}E(t;u_h,\tilde\zeta_{1},\tilde\zeta_{2})& = & 
\int_{0}^{L}{ \Lambda (u_{h})_{xx} (u_{h})_{xxt} \,dx} + 
\int_{0}^{L}{\mu (u_h)_t (u_{h})_{tt} \,dx} \\ && +  
 M (u_h)_t(L) (u_h)_{tt}(L) + J (u_h)_{tx}(L) (u_h)_{ttx}(L) \\
 &&+  k_1 (u_h)_x(L)(u_h)_{xt}(L)+ k_2 (u_h)(L)(u_h)_t(L) \\  && + 
 \tilde\zeta_1^{\top} P_1 (\tilde\zeta_1)_t +  \tilde\zeta_2^{\top} P_2 (\tilde\zeta_2)_t
 \\  & = &  -d_1 (u_h)_{xt}(L)^2 -d_2 (u_h)_{t}(L)^2  \\ && -c_1 \cdot
\tilde\zeta_1 (u_h)_{xt}(L) -c_2 \cdot \tilde\zeta_2(u_h)_{t}(L) + 
\tilde\zeta_1^{\top} P_1 (\tilde\zeta_1)_t +  \tilde\zeta_2^{\top} P_2 (\tilde\zeta_2)_t, 
\end{eqnarray*}
and the result follows with \eqref{zeta-tilde-ODE} and \eqref{kyp}.
\end{proof}
In the undamped case (i.e.~$A_j=0, d_j=0$) the energy $E$ is clearly preserved in the semi-discrete system.
Furthermore, it has been shown in the proof of Theorem \ref{continuity_weak_sol}
 that the energy functional for the weak solution $\hat{u}$, $\zeta_{1}, \zeta_2$ of 
\eqref{weak-form} - \eqref{hat_ic} has an analogous dissipative property, cf.~\eqref{limit_epsilon_energy}.

\subsection{Error estimates: semi-discrete scheme}

Since using cubic polynomials for the space approximation, we shall obtain accuracy 
of order two in space (in $H^2(0, L)$). Thereby, the common method for obtaining error estimates (cf. \cite{Choo:Chung})
will be adjusted to the problem at hand. With $\tilde u$ we denote the nodal projection of the weak solution 
$u$ to $W_h$, defined in terms of Hermite polynomials:
\begin{eqnarray*}
\tilde u (t,x) & = & \sum_{m = 1}^{P}{u(t,x_m) w_{2m-1}(x)} + \sum_{m=1}^{P}{ u_x(t,x_m) w_{2m}(x)}.
\end{eqnarray*}
Assuming that 
\begin{equation} \label{regularity_sol}
\begin{array}{l}
u \in C([0, T]; \tilde{H}_0^4(0, L)), \\ 
u_t \in L^2(0, T; \tilde{H}_0^4(0, L)), \\
u_{tt} \in  L^2(0, T; \tilde{H}_0^2(0, L)),
\end{array}
\end{equation} it can be seen (e.g. in \cite{Brenner:Scott}, \cite{Choo:Chung}) 
that a.e. in $t$:
\begin{equation} \label{Hermit_estimates}
\begin{array}{r c l}
\| u - \tilde u \|_{H^2(0, L)} & \le & C h^2 \| u \|_{H^4(0, L)}, \\
\| u_t - \tilde u_t \|_{H^2(0, L)} & \le & C h^2 \| u_t \|_{H^4(0, L)},\\
\| u_{tt} - \tilde u_{tt} \|_{L^2(0, L)} & \le & C h^2 \| u_{tt} \|_{H^2(0, L)}.
\end{array}
\end{equation}
We define the error of the semi-discrete solution $(u_h, \tilde \zeta_1, \tilde \zeta_2)$ as
 $\epsilon_h := u_h - \tilde u \in W_h$ and $\zeta^{e}_i := \tilde\zeta_i- \zeta_i, \; i = 1, 2.$ 
Then using (\ref{fem})--(\ref{zeta-tilde-ODE}) we obtain 

\begin{equation*} \label{error_semi}
\begin{array}{ccl}
 && \int_{0}^{L}{\mu (\epsilon_{h})_{tt} w \,dx}   + 
\int_{0}^{L}{\Lambda (\epsilon_h)_{xx} w_{xx} \,dx}   \\
 & & + M (\epsilon_h)_{tt}(t,L)w(L) + J (\epsilon_h)_{xtt}(t,L) w_{x}(L) 
 \\ & & + k_1 (\epsilon_{h})_x(t,L) w_x(L) +  k_2 \epsilon_h(t,L) w(L)  \\ 
& & + d_1 (\epsilon_h)_{xt}(t,L) w_x(L) + d_2 (\epsilon_h)_{t}(t,L) w(L)  \\ 
& & + c_1 \cdot \zeta^{e}_1(t) \; w_x(L) + 
 c_2 \cdot \zeta^{e}_2(t) \; w(L)\\ 
&& =\int_{0}^{L}{\mu (u_{tt} - \tilde u_{tt} ) w \,dx}   + 
\int_{0}^{L}{\Lambda (u_{xx} - \tilde u_{xx}) w_{xx} \,dx} , %\\ & & 
\quad\quad\quad \forall w \in W_h, \; t > 0,
\end{array}
\end{equation*}
coupled to:
\begin{eqnarray*} \label{error_semi_zeta}
\begin{array}{rcl}
(\zeta^{e}_{1})_t(t)  & = & A_{1} \zeta^{e}_{1}(t) + 
b_{1}  (\epsilon_h)_{xt} (t, L), \\  
(\zeta^{e}_{2})_t(t)  & = & A_{2} \zeta^{e}_{2}(t) + b_{2} (\epsilon_h)_{t} (t, L).
\end{array}
\end{eqnarray*}
Using $w = (\epsilon_h)_t$ and proceeding as in the proof of Theorem \ref{th:semidiscrete-decay} we obtain 
\begin{equation} \label{semi_aux} \begin{array}{l c l}
\frac12 \frac{d }{dt} E(t; \epsilon_h, \zeta^{e}_1, \zeta^{e}_2 ) & \le &\int_{0}^{L}{ \mu  (u_{tt}
 - \tilde u_{tt} ) (\epsilon_h)_t \,dx}   + 
\int_{0}^{L}{\Lambda (u_{xx} - \tilde u_{xx}) (\epsilon_h)_{txx} \,dx}, \end{array}
\end{equation}for  a.e. $t \in [0, T]$.  Integrating (\ref{semi_aux}) in time, and performing partial integration, we get 
\begin{equation} \label{integrated}
\begin{array}{r c l}
 E(t ; \epsilon_h, \zeta^{e}_1, \zeta^{e}_2 )  &\le&    E(0; \epsilon_h(0), \zeta^{e}_1(0), \zeta^{e}_2(0) ) 
\\ &+&
2 \int_{0}^{t}{\int_{0}^{L}{\mu (u_{tt}(s,x) - \tilde{u}_{tt}(s,x))(\epsilon_h)_t(s,x) \, dx} \,ds}
\\ &+&
2 \int_{0}^{L}{\Lambda (u_{xx}(t,x) - \tilde{u}_{xx}(t,x))(\epsilon_h)_{xx}(t,x) \, dx} 
\\ &+&
2 \int_{0}^{L}{\Lambda (u_{xx}(0,x) - \tilde{u}_{xx}(0,x))(\epsilon_h)_{xx}(0,x) \, dx}
\\ &-&
2 \int_{0}^{t}{\int_{0}^{L}{\Lambda (u_{txx}(s,x) - \tilde{u}_{txx}(s,x))(\epsilon_h)_{xx}(s,x) \, dx} \,ds}.
\\
\end{array} 
\end{equation}
Applying Chauchy-Schwarz to \eqref{integrated} yields:
\begin{equation} \label{Cauchy_schwarz_estimates}
\begin{array}{r c l}
 E(t ; \epsilon_h, \zeta^{e}_1, \zeta^{e}_2 )  &\le&    E(0; \epsilon_h(0), \zeta^{e}_1(0), \zeta^{e}_2(0) ) 
\\ &+\mu_{max}&
  \left( \|u_{tt} - \tilde{u}_{tt}\|^2_{L^2(0, T; L^2(0, L))} + \int_{0}^{t}{\|(\epsilon_h)_t(s,.)\|^2_{L^2(0,L)} \,ds} \right)
\\ &+\Lambda_{max}&
 \left( 8 \| u_{xx}(t,.) - \tilde{u}_{xx}(t,.)\|_{L^2(0, L)}^2 + \frac{1}{8} \|(\epsilon_h)_{xx}(t,.) \|_{L^2(0, L)}^2
\right. \\ &&+
8  \| u_{xx}(0,.) - \tilde{u}_{xx}(0,.)\|_{L^2(0, L)}^2 + \frac{1}{8} \|(\epsilon_h)_{xx}(0,.) \|_{L^2(0, L)}^2
\\ && +\left.
 \|u_{t} - \tilde{u}_{t}\|^2_{L^2(0, T; H^2(0, L))} + \int_{0}^{t}{\|(\epsilon_h)_{xx}(s,.)\|^2_{L^2(0,L)} \,ds}\right),
\end{array} 
\end{equation}
where $\mu_{max} = \max_{x \in [0, L]}{\mu(x)}$ and $\Lambda_{max} = \max_{x \in [0, L]}{\Lambda(x)}$.
Next, we use \eqref{Hermit_estimates} to obtain:
\begin{equation} \label{estimate_II}
\begin{array}{r c l}
 && \frac{3}{4} E(t ; \epsilon_h, \zeta^{e}_1, \zeta^{e}_2 )  \le  \frac{5}{4}  E(0; \epsilon_h(0), \zeta^{e}_1(0), \zeta^{e}_2(0) )  
+ 2 \int_{0}^{t}{ E(s ; \epsilon_h, \zeta^{e}_1, \zeta^{e}_2 ) \,ds}
\\ && + C h^4 \left(\| u \|_{C([0, T]; H^4(0,L))}^2 + \|u_{t} \|^2_{L^2(0, T; H^4(0, L))} + \|u_{tt} \|^2_{L^2(0, T; H^2(0, L))} \right). 
\end{array} 
\end{equation}
Gronwall inequality applied to \eqref{estimate_II} gives:
\begin{equation} \label{Gronwall_estimate}
\begin{array}{l}
 E(t ; \epsilon_h, \zeta^{e}_1, \zeta^{e}_2 ) \le  C \left( \vphantom{E^{\frac12}} E(0; \epsilon_h(0), \zeta_{1e}(0), \zeta_{2e}(0))\right.  
\\ + \left. \vphantom{E^{\frac12}} h^4  \left(\| u \|^2_{C([0, T]; H^4(0,L))} +
 \|u_{t} \|^2_{L^2(0, T; H^4(0, L))} + \|u_{tt} \|^2_{L^2(0, T; H^2(0, L))} \right) \right).
\end{array} 
\end{equation}
Finally, we have:
\begin{theorem}
Assuming (\ref{regularity_sol}), the following error estimate of the semidiscrete solution holds:
\begin{eqnarray} \label{semi_estimate}
 && E(t; u_h - u, \tilde\zeta_1- \zeta_1, \tilde\zeta_2 - \zeta_2)^{\frac12} 
 \le  C \left(E(0; \epsilon_h(0), \zeta_{1e}(0), \zeta_{2e}(0))^{\frac12} \right.  
\nonumber \\ &  & + \left.\vphantom{E^{\frac12}} h^2 \left( \| u_{tt}\|_{L^2(0,T; H^2(0, L))}  +  \| u_{t}\|_{L^2(0,T; H^4(0, L))} + 
\| u \|_{C([0,T]; H^4(0, L))} \right) \right),
\end{eqnarray}
$0 \le t \le T$.
\end{theorem}
\begin{proof}
The result follows from \eqref{Hermit_estimates}, \eqref{Gronwall_estimate}, and the triangle inequality.
\end{proof}
\begin{comment}
\begin{remark} \label{optimal_ic}
The estimate in (\ref{semi_estimate}) motivates to choose the initial conditions as follows:
\begin{eqnarray*}
u_{h,0} & = & \tilde{u},\\
v_{h,0} & = & \tilde{u}_t,\\
\tilde{\zeta}_1(0) &=&  \zeta_{1,0},\\
\tilde{\zeta}_2(0) &=&  \zeta_{2,0}.
\end{eqnarray*}
This implies
$\epsilon_{h}\equiv 0$ and $\zeta^e_{1}, \zeta^e_{1} \equiv 0$, and hence 
$E(0; \epsilon_h(0), \zeta_{1e}(0), \zeta_{2e}(0))=0$.
\end{remark}
\end{comment}

%%%%%%%%%%%%%%%%%%%%%%%%%%%%%%%%%%%%%%%%%%%%%%%%%%%%%%%%%%%%%%%%%%%%%%%
\subsection{Fully discrete scheme: time discretization}
For the numerical solution to the ODE \eqref{vector-system} we first write it as a first order system 
and then use the Crank-Nicolson scheme, which is crucial for the dissipativity of the scheme. 
To this end we introduce $v_h := (u_h)_t$, and $\mathbb{V}:=\mathbb{U}_t=[\:V_1\:\:V_2\:\:...\:\:V_N\:]^\top$ 
is its representation in the basis $\{w_j\}$.
The solution of the system \eqref{fem}, \eqref{zeta-tilde-ODE} is then the vector
$z_h=[\:u_h\;\:v_h\;\:\tilde\zeta_1\;\:\tilde\zeta_2\:]^\top$.
In contrast to \S\ref{s2}, here we do not have to include the boundary traces 
$v_h(L),\,(v_h)_x(L)$: In the finite dimensional case both $u_h$ and $v_h$ are in $\tilde H_0^2(0,L)$.
In analogy to \S\ref{s2}, the natural norm of $z_h = z_h(t)$ is defined as 
\begin{eqnarray}\label{discr-norm}
\| z_h \|^2
 &:=& \frac{1}{2} \int_{0}^{L}{\Lambda (u_h)_{xx}^2 \,dx} + \frac{1}{2} 
\int_{0}^{L}{\mu v_h^2 \,dx} +  \frac{M}{2} v_h^2(L) + \frac{J}{2} (v_h)_{x}^2(L)\\
&&+\frac{k_1}{2}(u_h)_x^2(L) +\frac{k_2}{2}u_h^2(L)
+\frac12 \tilde\zeta_1^\top P_1 \tilde\zeta_1 +\frac12 \tilde\zeta_2^\top P_2 \tilde\zeta_2.
\nonumber
\end{eqnarray}
Let $\Delta t$ denote the time step and 
\[ t_n = n \Delta t, \forall n \in \{0, 1, \dots, S \},\] 
is the discretization of the time interval $[0, T],\; T = S \Delta t$.
For the solution of the fully discrete scheme at $t = t_n$, we shall use the notation $z^n = [u^n \: v^n \: \zeta_1^n \: \zeta_2^n]^{\top}$. 
And $\mathbb{U}^n, \mathbb{V}^n$ are the basis representations (in $\{ w_j \}_{j=1}^{N}$) of $u^n$ and $v^n$, respectively.
Furthermore, let the vector $\mathbb{C}^n$ be defined by:  \[ (\mathbb{C}^n)_j :=  c_1 \cdot \zeta_1^n \; (w_j)_x(L) + 
 c_2 \cdot \zeta_2^n \; w_j(L), \quad j = 1, \dots, N.\]
The Crank-Nicolson scheme for \eqref{vector-system}, \eqref{zeta-tilde-ODE} then reads:
\begin{eqnarray} \label{CN__1}
\frac{\mathbb{U}^{n+1} - \mathbb{U}^{n}}{\Delta t} & = & \frac{1}{2}(
\mathbb{V}^{n+1} + \mathbb{V}^{n}), \\ \nonumber
\frac{{\mathbb{A}} \mathbb{V}^{n+1} - {\mathbb{A}} 
\mathbb{V}^{n}}{\Delta t} & = & - \frac{1}{2}({\mathbb{K}} 
\mathbb{U}^{n+1} + {\mathbb{K}} \mathbb{U}^{n})  - 
\frac{1}{2}(\mathbb{B} \mathbb{V}^{n+1} + \mathbb{B} 
\mathbb{V}^{n}) \\  
\label{CN__2} & & - \frac{1}{2} 
(\mathbb{C}^{n+1} + \mathbb{C}^{n}),\\ 
\label{CN__3} \frac{\zeta^{n+1}_1 - \zeta^{n}_1}{\Delta t}
 & = & A_1 \frac{\zeta^{n+1}_1 + \zeta^{n}_1}{2} 
+ b_1 \frac{v^{n+1}_x(L) + v^{n}_x(L)}{2}, \\ 
%\frac{V_N^{n+1}+V_N^{n}}{2},\\
\label{CN__4} \frac{\zeta^{n+1}_2 - \zeta^{n}_2}{\Delta t} 
& = & A_2 \frac{\zeta^{n+1}_2 + \zeta^{n}_2}{2} + 
b_2 \frac{v^{n+1}(L) + v^{n}(L)}{2}.
%\frac{V_{N-1}^{n+1}+V_{N-1}^{n}}{2}.
\end{eqnarray}
In the chosen basis $\{w_j\}$, the last term of \eqref{CN__3}, \eqref{CN__4} reads
$\left(V_N^{n+1}+V_N^{n}\right)/2$ and $\left(V_{N-1}^{n+1}+V_{N-1}^{n}\right)/2$, respectively.
\begin{comment}
The system matrix for this linear system is sparse as well. It consists of four block matrices; 
the one in the upper left corner is $2 N\times 2 N$, and beside the main diagonal has 2 more diagonals above and 4 below it.
The matrix in upper right corner  has only two non-zero rows, and one in the lower left, only two non-zero columns.
Finally, the matrix in the lower right corner is an (2n-1)-diagonal, $2 n \times 2 n$ matrix. 
\end{comment}
Next, we show that this scheme dissipates the norm. The somewhat lengthy proof is deferred to the Appendix B.

\begin{theorem} \label{fd_case}
For $n \in \mathbb{N}_0$ it holds for the norm from \eqref{discr-norm}:
\begin{eqnarray*}
\| z^{n+1} \|^2  & = & \| z^{n} \|^2 -
 \Delta t \left\{ \delta_1 \left( 
\frac{u^{n+1}_x(L)-u^{n}_x(L)}{\Delta t} \right)^2 \right.\\ 
& + & \frac{1}{2} \left( q_1\cdot \frac{\zeta_1^{n+1} + \zeta_1^{n}}{2} + 
\tilde\delta_1 \frac{u^{n+1}_x(L) - u^{n}_x(L)}{\Delta t} \right)^2 \\ 
& + &  \delta_2 \left( \frac{u^{n+1}(L)-u^{n}(L)}
{\Delta t} \right)^2 + \frac{1}{2} \left( q_2 \cdot
\frac{\zeta_2^{n+1} + \zeta_2^{n}}{2} + \tilde\delta_2 
\frac{u^{n+1}(L) - u^{n}(L)}{\Delta t} \right)^2 \\ 
& + & \left.
\frac{\epsilon_1}{2}  \frac{(\zeta_1^{n+1} + \zeta_1^{n})^{\top}}{2} P_1 
\frac{\zeta_1^{n+1} + \zeta_1^{n}}{2} + \frac{\epsilon_2}{2}  
\frac{(\zeta_2^{n+1} + \zeta_2^{n})^{\top}}{2} P_2 \frac{\zeta_2^{n+1} +
 \zeta_2^{n}}{2} \vphantom{ \frac{u^{n+1}_x(L)-u^{n}_x(L)}{\Delta t}^2} \right\}.
\end{eqnarray*}
\end{theorem}
This decay of the norm is consistent (as $\Delta t\to0$) with the decay \eqref{norm-decay} 
for the continuous case, and with Theorem \ref{th:semidiscrete-decay}.
For the uncontrolled beam (i.e. $\Theta_1 = \Theta_2 = 0$), Theorem \ref{fd_case} shows that $\| z^n \|$ 
is constant in $n$. This motivates our choice of the Crank-Nicolson time discretization.
\begin{remark}
Note that the scheme \eqref{CN__1}--\eqref{CN__4} and the 
norm dissipation property from Theorem \ref{fd_case}
were written independently of the basis $\{w_j\}$. Hence, this decay property 
applies to any choice of the subspace $W_h\subset \tilde H_0^2(0,L)$. And the same remark applies 
to Theorem \ref{th:semidiscrete-decay}.
\end{remark}

\subsection{Error estimates: Fully discrete scheme}

In this subsection we shall need to assume additional regularity of the weak solutions $u$, $\zeta_1$ and $\zeta_2$, 
in order to estimate the error of the fully discrete case: Suppose that $u \in H^4(0,T;\tilde H_0^2(0,L))$ and
 $\zeta_1, \zeta_2 \in H^3(0, T;\mathbb{R}^n)$. Let us define $\breve u \in W_h$ 
to be the projection of the weak solution $u$, such that
$$a (\breve u(t), w_h) = a(u(t), w_h), \quad \quad \quad \forall w_h \in W_h,$$ $\forall t \in [0, T]$.
One easily verifies that it holds: $\breve u \in H^4(0, T; \tilde H^2_0(0, L))$, since the projection $u \mapsto \breve u$ 
is bounded in $\tilde H^2_0(0, L)$. Furthermore,  let $u^e := u - \breve u$ denote the error of the projection.
Assuming $u \in H^2(0, T ; \tilde H_0^4(0, L))$, we obtain the error estimates for $\breve u$ (cf. \cite{Strang:Fix}):
\begin{equation} \label{projection_estimates}
\begin{array}{r c l}
\| u^e \|_{H^2(0, L)} & \le & C h^2 \| u \|_{H^4(0, L)}, \\
\| u^e_t  \|_{H^2(0, L)} & \le & C h^2 \| u_t \|_{H^4(0, L)},\\
\| u^e_{tt} \|_{H^2(0, L)} & \le & C h^2 \| u_{tt} \|_{H^4(0, L)}.
\end{array}
\end{equation} 
Let $z(t_n) = [ u(t_n)\; v(t_n) \; \zeta_1(t_n) \; \zeta_2(t_n)]^{\top}$ and $z^n = [ u^n \; v^n \; \zeta_1^n \; \zeta_2^n]^{\top}$
denote the solution of the system and the solution of the fully discrete scheme at time $t = t_n$, respectively. 
Then we define the error by 
\begin{eqnarray*}
\epsilon^n & : = & u^n - \breve{u}(t_n), \\  \Phi^n & : =& v^n - \breve{u}_t(t_n), \\ \zeta_{e,i}^n & : = 
& \zeta_i^n - \zeta_i(t_n), \quad i =1, 2,
\end{eqnarray*} and  $z_e^n : = [ \epsilon^n \; \Phi^n \; \zeta^n_{e,1} \; \zeta^n_{e,2} ]^{\top},$ for every $n \in {0, 1, \dots, S}$.

We now give the second order error estimate (both in space and time) of the fully discrete scheme. The proof is deferred to Appendix B.
\begin{theorem} \label{error_full}
Assuming $u \in H^2(0, T; \tilde H_0^4(0, L)) \cap H^4(0, T; \tilde H_0^2(0, L))$ and \\ $\zeta_1, \zeta_2 \in H^3(0, T; \mathbb{R}^n)$, 
the following estimate holds:
\begin{eqnarray*}
\| z^n - z(t_{n})  \| & \le &  C \left[  \| z^0_e \| + h^2 \| u \|_{H^2(0, T ; H^4(0, L))}  + 
(\Delta t)^2 \left( \| u_{tt} \|_{L^2(0, T ; H^4(0, L))} \nonumber \right. \right. \\ &+& 
\left. \left. \| u_{tt} \|_{H^2(0, T ; H^2(0, L))} + \| (\zeta_{1})_{tt} \|_{H^1(0, T; \mathbb{R}^n)} + 
\| (\zeta_{2})_{tt} \|_{H^1(0, T; \mathbb{R}^n)} \right) \right]. 
\end{eqnarray*}
\end{theorem}

%%%%%%%%%%%%%%%%%%%%%%%%%%%%%%%%%%%%%%%%%%%%%%%%%%%%%%%%%%%%%%%%%%%%%%%%%%%%%
\section{Numerical Simulation}\label{S-simul}
In this chapter we verify the dissipativity of our numerical scheme for an example with the following coefficients: 
 $\mu = \Lambda = L = 1$, $M = J= 0.1$, $k_1= k_2 = 0.01$, and $d_1= d_2 =0.02$.
\begin{figure}[h]
\begin{center}
\includegraphics[trim = 15mm 0mm 10mm 10mm, clip, scale=0.5]{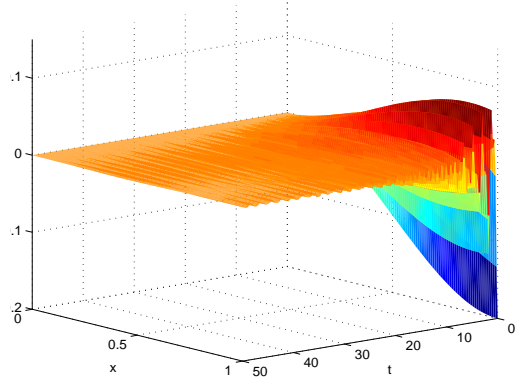}
\end{center}
\caption{Damped vibration of the beam: deflection $u(t,x)$}
%\vspace{3.5cm}
\label{Deflection}
\end{figure}
We take $n=10$ as the dimension 
of controller variables. Thereby, $A_1 = A_2 = - I \in \mathbb{R}^{10 \times 10}$, where $I$ is the identity matrix, and 
$b_1 = b_2 = c_1 = c_2 = [1 \; 1 \; \dots \; 1 ]^{\top}  \in \mathbb{R}^{10}$.
 We take the time step $\Delta t =0.01$ and the spatial discretization step $h=0.01$.
Figure \ref{Deflection} shows the damped oscillations of the beam and its convergence
to the steady state $u \equiv 0$ on the time interval $[0, 50]$. Figure \ref{Norm_evolution} illustrates the 
(slower then exponential) energy dissipation of the coupled control system.
\begin{figure}[h]
\begin{center}
\includegraphics[trim = 0mm 0mm 0mm 0mm,scale=0.5]{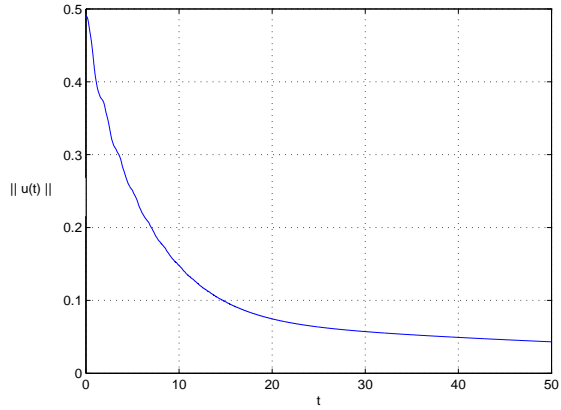}
\end{center}
%\vspace{5.5cm}
\caption{Dissipativity of the norm (or ``energy''): $\| z(t) \|_{\mathcal{H}}$}
\label{Norm_evolution}
\end{figure}
Finally, we perform simulations for different time and space discretization steps to verify the order 
of convergence (o.o.c.) proved in \S \ref{S-FEM}. In Table \ref{ooc_table} we list the $l^2$-error norms of $z_e$.
\begin{table}[h] 
\caption{Experimental convergence rates}
\label{ooc_table}
\begin{minipage}[b]{0.45\linewidth}
\begin{center}
\begin{tabular}{  c  c  c  c } 
\toprule
$\Delta t$ & $h$ & $\| z_e \|_{l^2} $ & o.o.c. \\ \midrule

$10^{-2}$ & $\frac{1}{4}$ & $1.75 * 10^{-2}$ & $--$  \\ 

$10^{-2}$ & $\frac{1}{8}$ & $5.5 * 10^{-3}$ & $1.67$ \\

$10^{-2}$ & $\frac{1}{16}$ & $7.92 * 10^{-4}$ & $2.80$ \\ 

$10^{-2}$ & $\frac{1}{32}$ & $1.39* 10^{-4}$ & $2.51$ \\ 

$10^{-2}$ & $\frac{1}{64}$ & $3.38 * 10^{-5}$ & $2.04$ \\ 

$10^{-2}$ & $\frac{1}{128}$ & $8.24 * 10^{-6}$ & $2.04$ \\
 \bottomrule \\
 \end{tabular}
\end{center}
\end{minipage}
\hspace{0.5cm}
\begin{minipage}[b]{0.45\linewidth}
\begin{center}
\begin{tabular}{  c  c  c  c }
\toprule
 $\Delta t$ & $h$ & $\| z_e \|_{l^2} $ & o.o.c. \\ \midrule

$6.4 * 10^{-6}$ & $\frac{1}{50}$ & $2.58 * 10^{-6}$ & $--$ \\

$3.2 * 10^{-6}$ & $\frac{1}{50}$ & $6.87 * 10^{-7}$ & $1.91$ \\

$1.6 * 10^{-6}$ & $\frac{1}{50}$ & $1.73 * 10^{-7}$ & $1.99$ \\

$8 * 10^{-7}$ & $\frac{1}{50}$ & $4.27 * 10^{-8}$ & $2.02$ \\

$4 * 10^{-7}$ & $\frac{1}{50}$ & $1.02 * 10^{-8}$ & $2.07$ \\ 

$2 * 10^{-7}$ & $\frac{1}{50}$ & $2.03 * 10^{-9}$ & $2.32$ \\
\bottomrule \\
\end{tabular}
\end{center}
\end{minipage}
\end{table}
In the left table we see the  o.o.c. results for fixed $\Delta t=0.01$ and varying space discretization 
step $h$ on the time interval $[0, 1]$. In the right table the  o.o.c. results for different 
$\Delta t$ but $h=1/50$ fixed, on the time interval $[0, 0.00041]$, are presented.
%%%%%%%%%%%%%%%%%%%%%%%%%%%%%%%%%%%%%%%%%%%%%%%%%%%%%%%%%%%%%%%%%%%%%%%%%
\section{Appendix A}

The following proof is an adaption of the proof of Theorem 8.1 in \cite{Lions:Magenes},
 for the system studied here. It is included for the sake of completeness.
\begin{proof}[\textbf{Proof of Theorem \ref{existence_weak_sol}}] 
\textit{(a)--existence}:
Let $\{ \hat{w}_k \}_{k=1}^{\infty}$ be a sequence of functions that is an orthonormal basis for $H$, 
and an orthogonal basis for $V$. We introduce $W_m := span \{ \hat{w}_1, \dots, \hat{w}_m\}, \forall m \in \mathbb{N}$. 
Furthermore, let sequences 
$\hat{u}_{m0}, \hat{v}_{m0} \in W_m$ be given so that 
\begin{equation} \label{initial_convergence}
\begin{split}
 \hat{u}_{m0} \rightarrow \hat{u}_0 \text{ in } V, \\
 \hat{v}_{m0} \rightarrow \hat{v}_0 \text{ in } H. \\
\end{split}
\end{equation}

For a fixed $m \in \mathbb{N}$ we consider the Galerkin approximation 
$$\hat{u}_m(t)= ((u_m)_x(L),u_m(L), u_m) = \sum_{k=1}^{m}{d_m^k(t) \hat{w}_k},$$
with $d_m^k(t) \in \mathbb{R}$, which solves the formulation (\ref{motivation-form}) for all $\hat{w} \in W_m$:
\begin{equation} \label{Galerkin-form}
 ((\hat{u}_m)_{tt}, \hat{w})_H + a(\hat{u}_m, \hat{w}) + b((\hat{u}_m)_t, \hat{w}) 
+ e_1(\zeta_{1,m}, \hat{w}) + e_2(\zeta_{2,m}, \hat{w})= 0, 
\end{equation}and $\zeta_{1,m},\zeta_{2,m}$ solve the ODE system 
\begin{eqnarray} \label{discrete-zeta-ODE}
\begin{array}{rcl}
(\zeta_{1,m})_t(t)  & = & A_{1} \zeta_{1,m}(t) + 
b_{1}  \,^1(\hat{u}_m)_{t} (t), \\  
(\zeta_{2,m})_t(t)  & = & A_{2} \zeta_{2,m}(t) + b_{2} \,^2(\hat{u}_m)_{t} (t),
\end{array}
\end{eqnarray}
 with the initial conditions 
\begin{eqnarray*}
\hat{u}_m(0) & = & \hat{u}_{m0} , \\
(\hat{u}_m)_t(0) & = & \hat{v}_{m0} , \\
\zeta_{1,m}(0) & = & \zeta_{0,1} , \\
\zeta_{2,m}(0) & = & \zeta_{0,2}.
\end{eqnarray*}

This problem is a linear system of second order differential equations, with a unique solution satisfying
 $\hat{u}_m \in C^2([0, T]; V)$ and $\zeta_{1,m}, \zeta_{2,m} \in C^1([0,T];\mathbb{R}^n)$. 
Next, we define an energy functional, analogous to (\ref{energy_functional}), for the \textit{trajectory} $(\hat{u}, \zeta_1, \zeta_2)$:
\begin{eqnarray*}
   \hat{E}(t; \hat{u}, \zeta_1, \zeta_2) & := & \frac{1}{2} \| \hat{u}(t) \|^2_V + \frac{k_1}{2} (^1\hat{u}(t))^2 
     + \frac{k_2}{2} (^2\hat{u}(t))^2  + \frac12 \| \hat{u}_t(t) \|^2_H \\ &&
      + \frac12 \zeta_{1}^{\top}(t) P_1 \zeta_{1}(t) + \frac12 \zeta_{2}^{\top}(t) P_2 \zeta_{2}(t) \\
 & = & \| (u, u_t, \zeta_1, \zeta_2, J u_{tx}(J), M u_t(L)) \|_{\mathcal{H}}.
\end{eqnarray*}
Taking $\hat{w} = (\hat{u}_m)_t$ in (\ref{Galerkin-form}) and using the 
smoothness of $\hat{u}_m, \zeta_{1,m}, \zeta_{2,m}$, a 
straightforward calculation yields 
\begin{eqnarray} \label{norm-diff}
\frac{d}{dt} \hat{E}(t;\hat{u}_m,\zeta_{1,m}, \zeta_{2,m}) & = &  - 
\delta_1 (^1(\hat{u}_m)_{t})^2 -\frac{1}{2} \left(\zeta_{1,m} \cdot q_1 + 
\tilde\delta_1 (^1(\hat{u}_m)_{t}) \right)^2  \nonumber \\ 
&  & - \delta_2 (^2(\hat{u}_m)_{t})^2 - \frac{1}{2} \left(\zeta_{2,m} \cdot q_2 + 
\tilde\delta_2(^2(\hat{u}_m)_{t})\right)^2 \nonumber \\ &  & - \frac{\epsilon_1}{2} (\zeta_{1,m})^{\top} P_1 
\zeta_{1,m}- \frac{\epsilon_2}{2} (\zeta_{2,m})^{\top} P_2 \zeta_{2,m}\nonumber \\
&=:& F(t;\hat{u}_m, \zeta_{1,m}, \zeta_{2,m}) \le 0, 
\end{eqnarray}
which is analogous to \eqref{norm-decay} for the continuous solution. Hence 
\begin{equation*} \label{energy-bound}
\hat{E}(t;\hat{u}_m,\zeta_{1,m}, \zeta_{2,m}) \le \hat{E}(0;\hat{u}_m,\zeta_{0,1}, \zeta_{0,2}), \quad t \ge 0, 
\end{equation*}
which implies

\begin{eqnarray} \label{boundedness_gal_solutions}
&\{\hat{u}_m \}_{m \in \mathbb{N}}  & \text{ is bounded in } C([0,T]; V), \nonumber\\
&\{(\hat{u}_m)_t \}_{m \in \mathbb{N}} &  \text{ is bounded in } C([0,T]; H), \\
&\{\zeta_{1,m} \}_{m \in \mathbb{N}},  \{\zeta_{2,m} \}_{m \in \mathbb{N}} & \text{ are bounded in } C([0,T]; \mathbb{R}^n).
\nonumber \\ && \nonumber
\end{eqnarray}

Due to these boundedness results, it holds $\forall \hat{w} \in V$:
\begin{eqnarray*}
 |a(\hat{u}_m(t), \hat{w}) + b((\hat{u}_m)_t(t), \hat{w}) + e_1(\zeta_{1,m}(t), \hat{w}) 
+ e_2(\zeta_{2,m}(t), \hat{w})| &\le& D_1 \| \hat{w} \|_{V}, 
\end{eqnarray*}
a.e. on $(0, T)$, with some constant $D_1 >0$ which does not depend on $m$. 
Now, let $m \in \mathbb{N}$ be fixed. 
Furthermore, let $\hat{w} \in V$, and $\hat{w} = \hat{w}_1 + \hat{w}_2$, 
such that $\hat{w}_1 \in W_m$ and $\hat{w}_2$ orthogonal to $W_m$ in $H$. Then
we obtain from (\ref{Galerkin-form}):
\begin{eqnarray*}
 ((\hat{u}_m)_{tt}, \hat{w})_H & = & ((\hat{u}_m)_{tt}, \hat{w}_1)_H \\ 
& = & -a(\hat{u}_m, \hat{w}_1) - b((\hat{u}_m)_t, \hat{w}_1) - e_1(\zeta_{1,m}, \hat{w}_1) 
- e_2(\zeta_{2,m}, \hat{w}_1)\\ & \le & D_1 \|\hat{w}_1 \|_V \le D_1 \|\hat{w} \|_V.
\end{eqnarray*}
This implies that also $(\hat{u}_m)_{tt}$ is bounded in $L^2(0,T; V')$. 
Furthermore, from (\ref{discrete-zeta-ODE}) it trivially follows that 
$\{(\zeta_{1,m})_t \}_{m \in \mathbb{N}}$ and $(\{\zeta_{2,m})_t \}_{m \in \mathbb{N}}$  
are also bounded in $L^{2}(0,T; \mathbb{R}^n)$.

According to the Eberlein-\u{S}muljan Theorem, there  exist subsequences 
$\{ \hat{u}_{m_l}\}_{l \in \mathbb{N}}$, $\{ \zeta_{1,m_l}\}_{l \in \mathbb{N}}$,
$\{ \zeta_{2,m_l}\}_{l \in \mathbb{N}}$, and $\hat{u} \in L^{2}(0,T ; V)$, with 
$\hat{u}_t \in L^{2}(0,T ; H)$, $\hat{u}_{tt} \in L^{2}(0,T ; V')$, and 
$\zeta_1, \zeta_2 \in H^1(0,T;\mathbb{R}^n)$ such that 

%\begin{equation} 
\begin{eqnarray}
\label{weak_convergence}
%\begin{align}
&\{\hat{u}_{m_l} \} \rightharpoonup u  \text{ in }L^{2}(0,T; V), \nonumber\\
&\{(\hat{u}_{m_l})_t \} \rightharpoonup u_t \text{ in }L^{2}(0,T; H), \nonumber\\
&\{(\hat{u}_{m_l})_{tt} \} \rightharpoonup u_{tt} \text{ in } L^{2}(0,T; V'), \nonumber\\
&\{\zeta_{1,m_l} \} \rightharpoonup \zeta_1  \text{ in }  L^{2}(0,T; \mathbb{R}^n),\\
&\{\zeta_{2,m_l} \} \rightharpoonup \zeta_2  \text{ in } L^{2}(0,T;\mathbb{R}^n), \nonumber\\
&\{(\zeta_{1,m_l})_t \} \rightharpoonup (\zeta_1)_t \text{ in }  L^{2}(0,T; \mathbb{R}^n),\nonumber\\
&\{(\zeta_{2,m_l})_t \} \rightharpoonup (\zeta_2)_t   \text{ in } L^{2}(0,T;\mathbb{R}^n). \nonumber
%\end{align}
%\end{equation}
\end{eqnarray}
Therefore, passing to the limit in (\ref{Galerkin-form}) and (\ref{discrete-zeta-ODE}), 
we see that $\hat{u}$ and $\zeta_1, \zeta_2$ solve 
(\ref{weak-form}) and (\ref{zeta-ODE}).
 
\textit{(b)}--additional regularity: {}From $\zeta_1, \zeta_2 \in H^1(0, T; \mathbb{R}^n)$ follows the continuity of the controller 
functions, i.e. (\ref{reg_zeta}). It is easily seen from the construction of the weak solution  and 
(\ref{boundedness_gal_solutions}) that $\hat{u}$ satisfies (\ref{reg_infty}). (\ref{reg_con_u}) follows immediately due to 
Lemma \ref{Continuity_Lions}, after, possibly, a modification on a set of measure zero. (\ref{reg_con_u_t}) follows from Lemma 
\ref{Continuity_Lions} and the 'Duality Theorem'
 (see \cite{Lions:Magenes}, Chapter 6.2, pp. 29) which states: for all $\theta \in (0, 1)$, it holds $$[X, Y]_{\theta}' = [Y', X']_{1 - \theta}.$$ 

\textit{(a)-initial conditions, uniqueness}: It remains to show that $\hat{u}$, $\zeta_1$, and $\zeta_2$ satisfy the initial conditions. For this purpose, 
we integrate by parts (in time) in (\ref{weak-form}), with $\hat{w} \in C^2([0, T]; V)$ such that $\hat{w}(T) = 0$ and $\hat{w}_t(T) = 0$:
\begin{equation} \label{init_solution}
\begin{split}
\int_{0}^{T}{
\left[ (\hat{u}, \hat{w}_{tt})_H + a(\hat{u}, \hat{w}) + b(\hat{u}_t, \hat{w}) + e_1(\zeta_1, \hat{w})
 + e_2(\zeta_2, \hat{w})\right] \, d \tau} 
= \\ -(\hat{u}(0), \hat{w}_t(0))_H + {_{V'}<\hat{u}_t(0), \hat{w}(0)>_{V}}.
\end{split}
\end{equation} 
Similarly, for a fixed $m$ it follows from (\ref{Galerkin-form}):
\begin{equation} \label{init_galerkin}
\begin{split}
\int_{0}^{T}{
\left[ (\hat{u}_m, \hat{w}_{tt})_H + a(\hat{u}_m, \hat{w}) + b((\hat{u}_m)_t, \hat{w}) +  e_1(\zeta_{1m}, \hat{w}) 
+  e_2(\zeta_{2m}, \hat{w}) \right] \, d \tau}=
\\ -(\hat{u}_{m0}, \hat{w}_t(0))_H + (\hat{v}_{m0}, \hat{w}(0))_H.
\end{split}
\end{equation}
Due to (\ref{initial_convergence}) and (\ref{weak_convergence}), 
passing to the limit in (\ref{init_galerkin}) along the convergent subsequence $\{\hat{u}_{m_l}\}$ gives
\begin{equation} \label{init_limit}
\begin{split}
\int_{0}^{T}{
\left[ (\hat{u}, \hat{w}_{tt})_H + a(\hat{u}, \hat{w}) + b(\hat{u}_t, \hat{w}) + e_1(\zeta_1, \hat{w}) 
+ e_2(\zeta_2, \hat{w}) \right] \, d \tau} = 
\\ -(\hat{u}_0, \hat{w}_t(0))_H + (\hat{v}_0, \hat{w}(0))_H.
\end{split}
\end{equation} 
Comparing (\ref{init_solution}) with (\ref{init_limit}), implies $\hat{u}(0) = \hat{u}_0$ and $\hat{u}_t(0) = \hat{v}_0$. 
Analogously we obtain $\zeta_1(0)=\zeta_{0,1}$ and $\zeta_2(0)=\zeta_{0,2}$.

In order to show uniqueness, let $(\hat{u}, \zeta_1, \zeta_2)$ be a solution to (\ref{weak-form})
and (\ref{zeta-ODE}) with zero initial conditions. Let $s \in (0, T)$ be fixed, and set
\begin{equation*}
\hat{U}(t):= \left\{
\begin{array}{c c l}
\int_{t}^{s}{\hat{u}(\tau) \,d\tau}, & & t < s, \\
0, & & t \ge s,
\end{array} \right.        
\end{equation*}
and
\begin{equation*}
Z_i(t):= 
\int_{0}^{t}{\zeta_i(\tau) \,d\tau},    
\end{equation*}
for $i = 1,2$. Integrating (\ref{zeta-ODE}) over $(0, t)$ yields with \eqref{kyp}
\begin{eqnarray} \label{uniq_1}
\frac12 \frac{d}{dt}(Z_i^{\top} P_i Z_i)(t) & = & - \frac12 \epsilon_i Z_i^{\top}(t) P_i Z_i(t) - 
\frac12 (q_i \cdot Z_i(t) + \tilde \delta_i (^i \hat{u}(t)))^2 \nonumber \\ & + &
(d_i - \delta_i)(^i \hat{u}(t))^2 + Z_i(t) \cdot c_i (^i \hat{u}(t)),
\end{eqnarray} for $0 \le t \le T$, $i = 1, 2$. Integrating (\ref{weak-form}) with $\hat{w} = \hat{U}$ over $[0, T]$,  and performing 
partial integration in time, yields
\begin{eqnarray} \label{uniq_2}
&& \int_{0}^{s}{(\hat{u}_{t}(\tau), \hat{u}(\tau))_{H} - a(\hat{U}_t(\tau), \hat{U}(\tau)) + b(\hat{u}(\tau), \hat{u}(\tau)) \,d\tau} 
\nonumber \\ && + \sum_{i = 1}^{2}  \int_{0}^{s}{Z_i(\tau) \cdot c_i (^i\hat{u}(\tau)) \,d\tau}= 0.
\end{eqnarray} 
{}From (\ref{uniq_1}) and (\ref{uniq_2}) follows 
\begin{eqnarray*} 
&& \int_{0}^{s}{\frac{d}{dt}\left(\frac12 \|\hat{u}(\tau)\|^2_{H} - \frac12 a(\hat{U}(\tau), \hat{U}(\tau))
+ \frac12 \sum_{i=1}^2{Z_i^{\top}(\tau) P_i Z_i(\tau)}\right)  \,d\tau} 
\nonumber \\ && = - \sum_{i = 1}^{2}  
 \int_{0}^{s}{\left( \delta_i (^i\hat{u}(\tau))^2 + \frac{\epsilon_i}{2} Z_i^{\top}(\tau)P_i Z_i(\tau) + 
\frac12 (q_i \cdot Z_i(\tau) + \tilde \delta_i (^i \hat{u})(\tau))^2\right)  \,d\tau}.
\end{eqnarray*}
Therefore, 
\begin{eqnarray*} 
\frac12 \|\hat{u}(s)\|^2_{H} + \frac12 a(\hat{U}(0), \hat{U}(0)) + \sum_{i=1}^2{\frac12 Z_i^{\top}(s) P_i Z_i(s)} & \le & 0.
\end{eqnarray*}
The matrices $P_j, j=1,2$ are positive definite, and the bilinear form $a(.,.)$ is coercive. Hence 
$\hat{u}(s) = 0$, $\hat{U}(0) = 0$, and $Z_i(s) = 0$. Since $s \in (0, T)$ was arbitrary, 
$\hat u \equiv 0$, $\zeta_i \equiv 0, \;i = 1, 2$ follows.
\end{proof}

%%%%%%%%%%%%%%%%%%%%%%%%%%%%%%%%%%%%%%%%%%%%%%%%%%%%%%%%%%%%%%%%%%%%%%%%%

Before the proof of the continuity in time of the weak solution, a definition and a lemma 
will be stated.
\begin{definition}
Let $Y$ be a Banach space. Then
\begin{eqnarray*}
 C_w([0, T]; Y) & := & \{ w \in L^{\infty}(0,T;Y) \colon \forall f \in Y' \\ & & t \mapsto (f,w(t)) 
\text{ is continuous on } [0, T]\}.
\end{eqnarray*} 
denotes the space of \emph{weakly continuous functions} with values in $Y$.
\end{definition}
The following Lemma was stated and proven in \cite{Lions:Magenes} (Chapter 8.4, pp. 275).
\begin{lemma} \label{weak_continuity}
Let $X$, $Y$ be Banach spaces, $X \subset Y$ with continuous injection, $X$ reflexive. Then
$$L^{\infty}(0, T; X) \cap C_w(0, T; Y) = C_w(0, T; X).$$
\end{lemma}

\begin{proof}[\textbf{Proof of Theorem \ref{continuity_weak_sol}}]
This proof is an adaption of standard strategies to the situation at hand (cf. \S8.4 in \cite{Lions:Magenes} and \S2.4 in \cite{Temam}).
Using Lemma \ref{weak_continuity} with $X = V$, $Y = H$,  we conclude from (\ref{reg_infty}), (\ref{reg_con_u}) that $\hat{u} \in C_w([0, T]; V)$.
Similarly, (\ref{reg_infty}) and (\ref{reg_con_u_t}) imply $\hat{u}_t \in C_w([0, T]; H)$.

Next, we take the scalar cut-off function $O_{I} \in C^{\infty}(\mathbb{R})$ such that it equals one on some interval $I \subset \subset [0, T]$,
and zero on $\mathbb{R} \setminus [0, T]$. 
Then the functions $O_{I} \hat{u} : \mathbb{R} \rightarrow V$ and $O_{I} \zeta_1,O_{I} \zeta_2  : \mathbb{R} \rightarrow \mathbb{R}^n$ 
are compactly supported.
Let $\eta^{\epsilon} : \mathbb{R} \rightarrow \mathbb{R}$ be a standard mollifier in time. Then we define
\begin{eqnarray*}
 \hat{u}^{\epsilon} &:=& \eta^{\epsilon} \ast O_{I} \hat{u} \in C^{\infty}_c(\mathbb{R}, V),\\
\zeta_1^{\epsilon} &:=& \eta^{\epsilon} \ast O_{I} \zeta_1 \in C^{\infty}_c(\mathbb{R}, \mathbb{R}^n),\\
\zeta_2^{\epsilon} &:=& \eta^{\epsilon} \ast O_{I} \zeta_2 \in C^{\infty}_c(\mathbb{R}, \mathbb{R}^n).
\end{eqnarray*}
Now $\zeta_1^{\epsilon}$ and $\zeta_2^{\epsilon}$ converge uniformly on $I$ to $\zeta_1$ and $\zeta_2$, respectively.
Moreover, $\hat{u}^{\epsilon}$ converges to $\hat{u}$ in $V$, and $\hat{u}^{\epsilon}_t$ to $\hat{u}_t$ in $H$ a.e. on $I$.
Then, $\hat{E}(t;\hat{u}^{\epsilon},\zeta_1^{\epsilon}, \zeta_2^{\epsilon})$ converges to $\hat{E}(t;\hat{u},\zeta_1, \zeta_2)$ a.e. on $I$ as well. 
Since $\hat{u}^{\epsilon}, \zeta_1^{\epsilon}, \zeta_2^{\epsilon}$ are smooth, a straightforward calculation on $I$ yields
\begin{eqnarray} \label{epsilon_energy}
\frac{d}{dt} \hat{E}(t;\hat{u}^{\epsilon},\zeta_1^{\epsilon}, \zeta_2^{\epsilon}) & = & F(t;\hat{u}^{\epsilon},\zeta_1^{\epsilon}, \zeta_2^{\epsilon}),
\end{eqnarray}
with $F$ defined in \eqref{norm-diff}. Passing to the limit in (\ref{epsilon_energy}) as $\epsilon \rightarrow 0$ 
\begin{eqnarray} \label{limit_epsilon_energy}
\frac{d}{dt} \hat{E}(t; \hat{u}, \zeta_1, \zeta_2) & = & F(t;\hat{u},\zeta_1, \zeta_2)
\end{eqnarray} holds in the sense of distributions on $I$. Since $I$ was arbitrary, (\ref{limit_epsilon_energy}) holds on all compact subintervals of $(0, T)$. 
Now $t \mapsto \hat{E}(t; \hat{u}, \zeta_1, \zeta_2)$ is an integral of an  $L^1$-function (note that the input 
functions of $F$ satisfy: $^1\hat{u}_t, ^2\hat{u}_t \in L^2(0, T)$), so it is absolutely continuous.

For a fixed $t$, let $\lim_{n \rightarrow +\infty}{t_n} = t$ and let the sequence $\chi_n$
be defined by 
\begin{eqnarray*}
   \chi_n & := & \frac{1}{2} \| \hat{u}(t) - \hat{u}(t_n)\|^2_V + \frac12 \| \hat{u}_t(t) - \hat{u}_t(t_n)\|^2_H \\ &&
     + \frac{k_1}{2} (^1\hat{u}(t)-\,^1\hat{u}(t_n))^2 + \frac{k_2}{2} (^2\hat{u}(t)-\,^2\hat{u}(t_n))^2  \\ &&
      + \frac12 (\zeta_{1}(t)-\zeta_{1}(t_n))^{\top} P_1 (\zeta_{1}(t) - \zeta_{1}(t_n)) \\ && + 
       \frac12 (\zeta_{2}(t)-\zeta_{2}(t_n))^{\top} P_2 (\zeta_{2}(t) - \zeta_{2}(t_n)).
\end{eqnarray*}
Then
\begin{eqnarray*}
   \chi_n & = & \hat{E}(t;\hat{u}, \zeta_1, \zeta_2) + \hat{E}(t_n;\hat{u}, \zeta_1, \zeta_2) - 
 (\hat{u}(t), \hat{u}(t_n))_V - (\hat{u}_t(t), \hat{u}_t(t_n))_H\\ &&
     - k_1 \,^1\hat{u}(t) ^1\hat{u}(t_n) - k_2 \,^2\hat{u}(t) ^2\hat{u}(t_n) 
      - \zeta_{1}(t)^{\top} P_1 \zeta_{1}(t_n) - \zeta_{2}(t) ^{\top} P_2 \zeta_{2}(t_n).
\end{eqnarray*}
Due to the $t$-continuity of the energy function, weak continuity of $\hat{u}, \hat{u}_t$, and continuity of $\zeta_1, \zeta_2$, it follows
$$\lim_{n \rightarrow +\infty}{\chi_n} = 0.$$ 
Finally, it follows that 
\begin{eqnarray*}
\lim_{n \rightarrow \infty}{ \| \hat{u}_t(t) - \hat{u}_t(t_n) \|^2_H} & = &0 , \\
\lim_{n \rightarrow \infty}{ \| \hat{u}(t) - \hat{u}(t_n)\|^2_V} &= &0,
\end{eqnarray*}
which proves the theorem.
\end{proof}

\section{Appendix B}

\begin{proof}[\textbf{Proof of Theorem \ref{fd_case}}]
First we obtain from (\ref{CN__1}) and (\ref{CN__2}) (written in the style of \eqref{motivation-form}):
\begin{eqnarray} 
& & \frac{u^{n+1} - u^{n}}{\Delta t}  = 
\frac{v^{n+1} + v^{n}}{2}, \label{aux1} 
\quad\quad\quad\quad\qquad  \\ [5mm]
& & \int^{L}_{0}
{\mu \frac{v^{n+1} - v^{n}}{\Delta t} w_h \,dx} +  
 \int^{L}_{0}{\Lambda \frac{u^{n+1}_{xx} + u^{n}_{xx}}{2} 
(w_{h})_{xx} \,dx} \nonumber \\ 
& & + M \frac{v^{n+1}(L) - v^{n}(L)}{\Delta t} w_h(L)  
+  J  \frac{v^{n+1}_{x} (L) - v^{n}_{x} (L)}{\Delta t}
 (w_h)_{x}(L) \nonumber \\ 
& & + k_1 \frac{u^{n+1}_x(L) +
 u^{n}_x(L) }{2} (w_h)_x(L) + k_2 \frac{u^{n+1}(L) +
 u^{n}(L) }{2} w_h(L) \label{aux2} \\ 
\nonumber & & + d_1 \frac{v^{n+1}_x(L) +
 v^{n}_x(L) }{2} (w_h)_x(L) + d_2 \frac{v^{n+1}(L) +
 v^{n}(L) }{2} w_h(L) \\ 
& & +  c_1 \cdot\frac{\zeta_1^{n+1} + \zeta_1^{n}}{2} (w_h)_x(L) 
+ c_2 \cdot\frac{\zeta_2^{n+1} + \zeta_2^{n}}{2} w_h(L) = 0,
\qquad\forall w_h \in W_h.\nonumber 
\end{eqnarray}
Next we multiply (\ref{aux1}) by  $\mu (v^{n+1} - v^n)$, and integrate over $[0,L]$ to obtain
$$\frac{1}{2} \int_{0}^{L}{\mu \left[(v^{n+1})^2 - (v^{n})^2 \right]\,dx} = 
 \int_{0}^{L}{\mu \frac{u^{n+1} - u^{n}}{\Delta t} (v^{n+1} - v^{n}) \,dx},
$$
and $w_h = u^{n+1}$ in (\ref{aux2}):
\begin{eqnarray*}
&&\frac{1}{2} \int_{0}^{L}{\Lambda(u^{n+1}_{xx})^2 \,dx} = 
- \frac{1}{2} \int_{0}^{L}{\Lambda u^{n+1}_{xx} u^{n}_{xx} \,dx} 
-  \int_{0}^{L}{\mu \frac{v^{n+1} - v^{n}}{\Delta t} u^{n+1} \,dx} \\ 
& &\qquad\qquad - M \frac{v^{n+1}(L) - v^{n}(L)}{\Delta t} u^{n+1}(L) 
- J \frac{v^{n+1}_x(L) - v^{n}_x(L)}{\Delta t} u^{n+1}_x(L)\\ 
 & &\qquad\qquad - k_1 \frac{u^{n+1}_x(L) + u^{n}_x(L)}{2}  u^{n+1}_x(L) 
- k_2 \frac{u^{n+1}(L) + u^{n}(L)}{2} u^{n+1}(L) \\
  & &\qquad\qquad -d_1 \frac{v^{n+1}_x(L) + v^{n}_x(L)}{2} u^{n+1}_x(L) 
- d_2 \frac{v^{n+1}(L) + v^{n}(L)}{2} u^{n+1}(L) \\
& &\qquad\qquad  - c_1 \cdot\frac{\zeta_1^{n+1} + \zeta_1^{n}}{2} u^{n+1}_x(L) 
-  c_2 \cdot\frac{\zeta_2^{n+1} + \zeta_2^{n}}{2} u^{n+1}(L).
\end{eqnarray*}
We next set $w_h=u^n$ in 
(\ref{aux2}):
\begin{eqnarray*}
&&\frac{1}{2} \int_{0}^{L}{\Lambda(u^{n}_{xx})^2 \,dx} =
- \frac{1}{2} \int_{0}^{L}{\Lambda u^{n+1}_{xx} u^{n}_{xx} \,dx} 
-  \int_{0}^{L}{\mu \frac{v^{n+1} - v^{n}}{\Delta t} u^{n}\,dx} \\
&&\qquad\qquad - M \frac{v^{n+1}(L) - v^{n}(L)}{\Delta t} u^{n}(L) 
-J \frac{v^{n+1}_x(L) - v^{n}_x(L)}{\Delta t} u^{n}_x(L) \\  
&&\qquad\qquad - k_1 \frac{u^{n+1}_x(L) + u^{n}_x(L)}{2} 
 u^{n}_x(L) - k_2 \frac{u^{n+1}(L) + u^{n}(L)}{2} u^{n}(L) \\  
&&\qquad\qquad - d_1 \frac{v^{n+1}_x(L) + v^{n}_x(L)}{2} u^{n}_x(L) 
- d_2 \frac{v^{n+1}(L) + v^{n}(L)}{2} u^{n}(L)\\
&&\qquad\qquad - c_1 \cdot\frac{\zeta_1^{n+1} + \zeta_1^{n}}{2} u^{n}_x(L) -
 c_2 \cdot\frac{\zeta_2^{n+1} + \zeta_2^{n}}{2} u^{n}(L) .
\end{eqnarray*}

This yields for the norm of the time-discrete solution, as defined in \eqref{discr-norm}:
\begin{eqnarray*}
&& \!\!\!\!\!\!\!\!\!\!\!\!\!\!\! \| z^{n+1} \|^2 - \| z^{n} \|^2 \\ 
& = & M \left( - \frac{v^{n+1}(L)
 - v^{n}(L)}{\Delta t} (u^{n+1}(L) - u^{n}(L)) + 
\frac{v^{n+1}(L)^2 - v^{n}(L)^2}{2} \right) \\ 
& + & J \left(- \frac{v^{n+1}_x(L) - v^{n}_x(L)}{\Delta t} (u^{n+1}_x(L) 
- u^{n}_x(L))  + \frac{v^{n+1}_x(L)^2 - v^{n}_x(L)^2}{2} \right) \\ 
& + & \frac{k_1}{2} \left(- \left(u^{n+1}_x(L) + 
u^{n}_x(L)\right) (u^{n+1}_x(L) - u^{n}_x(L))  + 
u^{n+1}_x(L)^2 - u^{n}_x(L)^2\right) \\ 
& + & \frac{k_2}{2} 
\left( - \left(u^{n+1}(L) + u^{n}(L)\right) (u^{n+1}(L) - 
u^{n}(L)) + u^{n+1}(L)^2 - u^{n}(L)^2 \right)  \\ 
& - & \frac{d_1}{2}  (v^{n+1}_x(L) + v^{n}_x(L)) 
(u^{n+1}_x(L)-u^{n}_x(L))\\ & - & \frac{d_2}{2} 
(v^{n+1}(L) + v^{n}(L))(u^{n+1}(L) - u^{n}(L)) \\ 
& -&  \frac{1}{2} c_1 \cdot(\zeta_1^{n+1} + \zeta_1^{n})(u^{n+1}_x(L) -
 u^{n}_x(L)) + \frac{1}{2}(\zeta_1^{n+1})^{\top} P_1 \zeta_1^{n+1} - 
\frac{1}{2}(\zeta_1^{n})^{\top} P_1 \zeta_1^{n} \\ 
&-&   \frac{1}{2} c_2 \cdot(\zeta_2^{n+1} + \zeta_2^{n})(u^{n+1}(L) - u^{n}(L)) + 
\frac{1}{2}(\zeta_2^{n+1})^{\top} P_2 \zeta_2^{n+1} - \frac{1}{2}(\zeta_2^{n})
^{\top} P_2 \zeta_2^{n} .
\end{eqnarray*}
For the first six lines we use (\ref{CN__1}), and for the rest $c_j=P_jb_j+q_j\tilde\delta_j$ (cf.\ (\ref{kyp})) to obtain:
\begin{eqnarray} \label{fully_discrete_norm_I}
 \| z^{n+1} \|^2  & = & \| z^{n} \|^2  -  
\frac{d_1}{\Delta t} \left(u^{n+1}_x(L)-u^{n}_x(L)\right)^2 
-  \frac{d_2}{\Delta t}(u^{n+1}(L) - u^{n}(L))^2 \nonumber \\& - &
 \frac{\left( \zeta_1^{n+1} + \zeta_1^{n} \right)^{\top}}{2} (P_1 b_1 + q_1 
\tilde\delta_1)(u^{n+1}_x(L) - u^{n}_x(L)) \nonumber \\& - &
 \frac{\left( \zeta_2^{n+1} + \zeta_2^{n} \right)^{\top}}{2} (P_2 b_2 + q_2 
\tilde\delta_2)(u^{n+1}(L) - u^{n}(L)) \nonumber \\ & + & \frac{1}{2} 
(\zeta_1^{n+1})^{\top} P_1 \zeta_1^{n+1} - \frac{1}{2} (\zeta_1^{n})^{\top}
 P_1 \zeta_1^{n} + \frac{1}{2} (\zeta_2^{n+1})^{\top} P_2 \zeta_2^{n+1} -
 \frac{1}{2} (\zeta_2^{n})^{\top} P_2 \zeta_2^{n}.\nonumber\\
\end{eqnarray}
For the second and the third line of \eqref{fully_discrete_norm_I} 
we now use (\ref{CN__1}), (\ref{CN__3}), and (\ref{CN__4}) from the Crank-Nicholson scheme:
\begin{eqnarray*}
 \| z^{n+1} \|^2 & = & \| z^{n} \|^2  -  
\frac{d_1}{\Delta t} 
\left(u^{n+1}_x(L)-u^{n}_x(L)\right)^2 -  
\frac{d_2}{\Delta t}(u^{n+1}(L) - u^{n}(L))^2 \\
& - & \frac{\left( \zeta_1^{n+1} + \zeta_1^{n} \right)^{\top}}{2}
 P_1 \left(\zeta_1^{n+1} - \zeta_1^{n} - \Delta t \; A_1 \frac{\zeta_1^n
 + \zeta_1^{n+1}}{2}\right) \\& - & \frac{\left( \zeta_1^{n+1} + \zeta_1^{n}
 \right)}{2} \cdot q_1 \tilde\delta_1 (u^{n+1}_x(L) - u^{n}_x(L)) \\
& - & \frac{\left( \zeta_2^{n+1} + \zeta_2^{n} \right)
^{\top}}{2} P_2 \left(\zeta_2^{n+1} - \zeta_2^{n} - \Delta t \; A_2 
\frac{\zeta_2^{n+1} + \zeta_2^{n}}{2} \right) \\& - & \frac{\left( 
\zeta_2^{n+1} + \zeta_2^{n} \right)}{2} \cdot q_2 \tilde\delta_2
(u^{n+1}(L) - u^{n}(L)) \\ & + & \frac{1}{2} (\zeta_1^{n+1})^
{\top} P_1 \zeta_1^{n+1} - \frac{1}{2} (\zeta_1^{n})^{\top} P_1 \zeta_1^{n} 
 + \frac{1}{2} (\zeta_2^{n+1})^{\top} P_2 \zeta_2^{n+1} - \frac{1}{2} 
(\zeta_2^{n})^{\top} P_2 \zeta_2^{n}.
\end{eqnarray*}
Since $P_j, \,j=1,2$ are symmetric matrices, this yields
\begin{eqnarray*}
  \| z^{n+1} \|^2 & = & \| z^{n} \|^2  -  
\frac{d_1}{\Delta t} 
\left(u^{n+1}_x(L)-u^{n}_x(L)\right)^2 - 
 \frac{d_2}{\Delta t}(u^{n+1}(L) - u^{n}(L))^2 \\
& + & \Delta t\frac{\left( \zeta_1^{n+1} + \zeta_1^{n} \right)^{\top}}{2}
 P_1  A_1 \frac{\zeta_1^n + \zeta_1^{n+1}}{2} \\ & - & \frac{\left( 
\zeta_1^{n+1}
 + \zeta_1^{n} \right)}{2} \cdot q_1 \tilde\delta_1
(u^{n+1}_x(L) - u^{n}_x(L)) \\
& + & \Delta t 
\frac{\left( \zeta_2^{n+1} + \zeta_2^{n} \right)^{\top}}{2} P_2 
A_2 \frac{\zeta_2^{n+1} + \zeta_2^{n}}{2} \\ 
& - & \frac{\left( \zeta_2^{n+1}
 + \zeta_2^{n} \right)}{2} \cdot q_2 \tilde\delta_2 (u^{n+1}(L) - u^{n}(L)) ,
\end{eqnarray*} 
which is the claimed result (by using (\ref{kyp})).
\end{proof}
\begin{proof}[\textbf{Proof of Theorem \ref{error_full}}]
Let $k \in \{0, 1, \dots, S  \}$ be arbitrary. Taylor's Theorem yields
  $\forall x \in [0, L]$:
\begin{eqnarray} \label{first_expansion}
\frac{\breve{u}(t_{k+1},x)-\breve{u}(t_{k},x)}{\Delta t} & = & \frac{\breve{u}_{t}(t_{k+1},x) 
+ \breve{u}_{t}(t_{k},x)}{2} + \Delta t \, T^k_1(x),
\end{eqnarray}
where
\begin{eqnarray*} 
T^k_1 (x) & = & \int_{t_{k + \frac{1}{2}}}^{t_{k+1}}{\frac{\breve u_{ttt}(t, x)}{2 \left( \Delta t \right)^2} 
(t_{k+1} - t)^2 \;dt} +
\int_{t_{k}}^{t_{k+\frac12}}{\frac{\breve u_{ttt}(t, x)}{2 \left( \Delta t \right)^2} (t_{k} - t)^2 \;dt} \\
& - & \int_{t_{k + \frac{1}{2}}}^{t_{k+1}}{\frac{\breve u_{ttt}(t, x)}{2 \Delta t} (t_{k+1} - t) \;dt} +
\int_{t_{k}}^{t_{k+\frac12}}{\frac{\breve u_{ttt}(t, x)}{2 \Delta t} (t_{k} - t)\;dt}.
\end{eqnarray*} 
{}From \eqref{first_expansion}, we obtain
\begin{eqnarray} \label{fd_formula_1}
 \frac{\epsilon^{k+1} - \epsilon^{k}}{\Delta t} + \Delta t \, T_1^k = \frac{\Phi^{k+1} + \Phi^{k}}{2}.
\end{eqnarray}
Multiplying \eqref{fd_formula_1} by $\mu (\Phi^{k+1} - \Phi^{k})$ and integrating over $[0, L]$ yields:
\begin{eqnarray} \label{fd_norm_part1}
&& \int_0^L{ \mu \frac{\epsilon^{k+1} - \epsilon^{k}}{\Delta t} \left(\Phi^{k+1} - \Phi^{k}\right) \;dx} \nonumber\\ 
&&= \frac{1}{2} \int_0^L{ \mu \left(\Phi^{k+1}\right)^2 \, dx} 
- \frac{1}{2} \int_0^L{ \mu \left(\Phi^{k}\right)^2 \, dx}- 
\Delta t  \int_0^L{ \mu  T_1^k \left(\Phi^{k+1} - \Phi^{k}\right)\;dx }. \nonumber \\
\end{eqnarray}
Furthermore, from \eqref{motivation-form} with $t = t_{k + \frac12}$ and Taylor's Theorem, we get $ \forall w \in \tilde{H}^2_0(0, L)$: 

\begin{align} \label{second_expansion}
 &\int^{L}_{0}
{\mu  \frac{u_t(t_{k+1},x) - u_t(t_{k},x)}{\Delta t} w \,dx} + 
 \int^{L}_{0}{ \Lambda \frac{u_{xx}(t_{k+1},x) + u_{xx}(t_{k},x)}{2} 
w_{xx} \,dx}&\nonumber  \\ 
 &+ M \frac{u_t(t_{k+1}, L) - u_t(t_{k}, L)}{\Delta t} w(L)  
+  J  \frac{u_{tx}(t_{k+1}, L) - u_{tx}(t_{k}, L)}{\Delta t}
w_{x}(L) & \nonumber \\ 
 &+ k_1 \frac{u_{x}(t_{k+1}, L) +
 u_{x}(t_{k}, L) }{2} w_x(L) + k_2 \frac{u(t_{k+1}, L) +
 u(t_{k}, L) }{2} w(L)&\nonumber \\ 
& + d_1 \frac{u_{tx}(t_{k+1}, L) +
u_{tx}(t_{k}, L)}{2} w_x(L) + d_2 \frac{u_{t}(t_{k+1}, L) +
 u_{t}(t_{k}, L) }{2} w(L) &\nonumber \\ 
& +  c_1 \cdot\frac{\zeta_1(t_{k+1}) + \zeta_1(t_k)}{2} w_x(L)
+ c_2 \cdot\frac{\zeta_2(t_{k+1}) + \zeta_2(t_k)}{2} w(L) = \Delta t \, T^k_2 (w), & \nonumber \\
&&
\end{align}
with the functional $T^k_2 \colon \tilde{H}^2_0(0, L) \rightarrow \mathbb{R}$ defined as
\begin{align} 
&T^k_2(w) = &\nonumber \\ &\int^{L}_{0}
{\mu \left( \int_{t_{k+\frac12}}^{t_{k+1}}{\frac{u_{tttt}(t,x)}{2 (\Delta t)^2} \left( t_{k+1} - t \right)^2\,dt }  + 
\int_{t_{k}}^{t_{k+\frac12}}{\frac{u_{tttt}(t,x)}{2 (\Delta t)^2} \left( t_{k} - t \right)^2\,dt } \right) w \,dx}  &\nonumber \\ 
 &  + \int^{L}_{0}{\Lambda  \left( \int_{t_{k+\frac12}}^{t_{k+1}}{\frac{u_{ttxx}(t,x)}{2 \Delta t} \left( t_{k+1} - t \right)\,dt }  - 
\int_{t_{k}}^{t_{k+\frac12}}{\frac{u_{ttxx}(t,x)}{2 \Delta t} \left( t_{k} - t \right)\,dt } \right) w_{xx} \,dx} & \nonumber \\ 
& + M \left( \int_{t_{k+\frac12}}^{t_{k+1}}{\frac{u_{tttt}(t,L)}{2 (\Delta t)^2} \left( t_{k+1} - t \right)^2\,dt }  + 
\int_{t_{k}}^{t_{k+\frac12}}{\frac{u_{tttt}(t,L)}{2 (\Delta t)^2} \left( t_{k} - t \right)^2 \,dt } \right) w(L) &  \nonumber \\ 
& + J \left( \int_{t_{k+\frac12}}^{t_{k+1}}{\frac{u_{ttttx}(t,L)}{2 (\Delta t)^2} \left( t_{k+1} - t \right)^2 \,dt }  + 
\int_{t_{k}}^{t_{k+\frac12}}{\frac{u_{ttttx}(t,L)}{2 (\Delta t)^2} \left( t_{k} - t \right)^2\,dt } \right) w_x(L)& \nonumber \\ 
& + k_1 \left( \int_{t_{k+\frac12}}^{t_{k+1}}{\frac{u_{ttx}(t,L)}{2 \Delta t} \left( t_{k+1} - t \right)\,dt }  - 
\int_{t_{k}}^{t_{k+\frac12}}{\frac{u_{ttx}(t,L)}{2 \Delta t} \left( t_{k} - t \right)\,dt } \right) w_x(L) & \nonumber \\ 
&  + k_2 \left( \int_{t_{k+\frac12}}^{t_{k+1}}{\frac{u_{tt}(t,L)}{2 \Delta t} \left( t_{k+1} - t \right) \,dt }  -
\int_{t_{k}}^{t_{k+\frac12}}{\frac{u_{tt}(t,L)}{2 \Delta t} \left( t_{k} - t \right)\,dt } \right) w(L) & \nonumber 
\end{align}
\begin{align} \label{T_2_def}
&  + d_1 \left( \int_{t_{k+\frac12}}^{t_{k+1}}{\frac{u_{tttx}(t,L)}{2 \Delta t} \left( t_{k+1} - t \right)\,dt }  -
\int_{t_{k}}^{t_{k+\frac12}}{\frac{u_{tttx}(t,L)}{2 \Delta t} \left( t_{k} - t \right)\,dt } \right) w_x(L) & \nonumber \\ 
&  + d_2 \left( \int_{t_{k+\frac12}}^{t_{k+1}}{\frac{u_{ttt}(t,L)}{2 \Delta t} \left( t_{k+1} - t \right)\,dt }  - 
\int_{t_{k}}^{t_{k+\frac12}}{\frac{u_{ttt}(t,L)}{2 \Delta t} \left( t_{k} - t \right) \,dt } \right) w(L) & \nonumber \\  
&  + c_1 \cdot \left( \int_{t_{k+\frac12}}^{t_{k+1}}{\frac{(\zeta_1)_{tt}(t)}{2 \Delta t} \left( t_{k+1} - t \right)\,dt }  - 
\int_{t_{k}}^{t_{k+\frac12}}{\frac{(\zeta_1)_{tt}(t)}{2 \Delta t} \left( t_{k} - t \right)\,dt } \right) w_x(L) & \nonumber \\  
&  + c_2 \cdot  \left( \int_{t_{k+\frac12}}^{t_{k+1}}{\frac{(\zeta_2)_{tt}(t)}{2 \Delta t} \left( t_{k+1} - t \right)\,dt }  - 
\int_{t_{k}}^{t_{k+\frac12}}{\frac{(\zeta_2)_{tt}(t)}{2 \Delta t} \left( t_{k} - t \right)\,dt } \right) w(L). & \nonumber \\ &&
\end{align}
Now, from \eqref{CN__2} and \eqref{second_expansion} follows $\forall w_h \in W_h$:
\begin{equation} \label{fd_formula_2}
\begin{array}{l}
  \int^{L}_{0}
{\mu \frac{\Phi^{k+1}- \Phi^{k}}{\Delta t} w_h \,dx} +  
 \int^{L}_{0}{\Lambda \frac{\epsilon_{xx}^{k+1} + \epsilon_{xx}^k}{2} 
(w_h)_{xx} \,dx}  \\ 
 + M \frac{\Phi^{k+1}(L) - \Phi^{k}(L)}{\Delta t} (w_h)(L)  
+  J  \frac{\Phi^{k+1}_x(L) - \Phi^{k}_x(L)}{\Delta t}
(w_h)_{x}(L)  \\ 
 + k_1 \frac{\epsilon^{k+1}_x(L) + \epsilon^{k}_x(L)}{2} (w_h)_x(L) + 
k_2 \frac{\epsilon^{k+1}(L) + \epsilon^{k}(L) }{2} w_h(L) \\ 
 + d_1 \frac{\Phi^{k+1}_x(L) + \Phi^{k}_x(L)}{2} (w_h)_x(L) + 
d_2 \frac{\Phi^{k+1}(L) +  \Phi^{k}(L)}{2} w_h(L) \\ 
+  c_1 \cdot\frac{\zeta_{e,1}^{k+1} + \zeta_{e,1}^k}{2} (w_h)_x(L)
+ c_2 \cdot\frac{\zeta_{e,2}^{k+1} + \zeta_{e,2}^k}{2} w_h(L) \\ 
= -\Delta t \, T^k_2 (w_h) + G^{k}_1(w_h), 
\end{array}
\end{equation}
where the functional $G^{k}_1(w_h)$ is given by
\begin{equation} \label{G_n1}
\begin{array}{l}
G^{k}_1(w_h) := \int^{L}_{0}
{\mu  \frac{u^e_t(t_{k+1},x) - u^e_t(t_{k},x)}{\Delta t} w_h \,dx}   \\ 
 + M \frac{u^e_t(t_{k+1}, L) - u^e_t(t_{k}, L)}{\Delta t} w_h(L)  
+  J  \frac{u^e_{tx}(t_{k+1}, L) - u^e_{tx}(t_{k}, L)}{\Delta t}
(w_h)_{x}(L) \\ 
 + d_1 \frac{u^e_{tx}(t_{k+1}, L) +
u^e_{tx}(t_{k}, L)}{2} (w_h)_x(L) + d_2 \frac{u^e_{t}(t_{k+1}, L) +
 u^e_{t}(t_{k}, L) }{2} w_h(L). 
\end{array}
\end{equation}
A Taylor expansion of $\zeta_j$ about $t_{k+\frac12}$ yields with \eqref{zeta-ODE}:
\begin{equation} \label{third_expansion}
\begin{array}{l}
\frac{\zeta_1(t_{k+1}) - \zeta_1(t_{k})}{\Delta t} - A_1 \frac{\zeta_1(t_{k+1}) + \zeta_1(t_{k})}{2} - b_1 \frac{u_{tx}(t_{k+1}, L) 
+ u_{tx}(t_{k}, L)}{2} = \Delta t \, T_3^k, \\
\frac{\zeta_2(t_{k+1}) - \zeta_2(t_{k})}{\Delta t} - A_2 \frac{\zeta_2(t_{k+1}) + \zeta_2(t_{k})}{2} - b_2 \frac{u_{t}(t_{k+1}, L) 
+ u_{t}(t_{k}, L)}{2} = \Delta t \, T_4^k ,
\end{array}
\end{equation}
with
\begin{eqnarray}
T_3^k & = & \int_{t_{k+\frac12}}^{t_{k+1}}{\frac{(\zeta_1)_{ttt}(t)}{2 (\Delta t)^2} \left( t_{k+1} - t \right)^2\,dt } 
+ \int_{t_{k}}^{t_{k+\frac12}}{\frac{(\zeta_1)_{ttt}(t)}{2 (\Delta t)^2} \left( t_{k} - t \right)^2\,dt } \nonumber \\
 &  & - A_1 \left( \int_{t_{k + \frac12}}^{t_{k +1 }}{\frac{(\zeta_1)_{tt}(t)}{2 \Delta t} \left( t_{k+1} - t \right)\,dt } 
- \int_{t_{k}}^{t_{k+\frac12}}{\frac{(\zeta_1)_{tt}(t)}{2 \Delta t} \left( t_{k} - t \right)\,dt }\right) \nonumber \\ 
& & - b_1 \left( \int_{t_{k + \frac12}}^{t_{k+1}}{\frac{u_{tttx}(t,L)}{2 \Delta t} \left( t_{k+1} - t \right)\,dt }  - 
\int_{t_{k}}^{t_{k+\frac12}}{\frac{u_{tttx}(t, L)}{2 \Delta t} \left( t_{k} - t \right)\,dt }\right), \nonumber \\
T_4^k & = & \int_{t_{k+\frac12}}^{t_{k+1}}{\frac{(\zeta_2)_{ttt}(t)}{2 (\Delta t)^2} \left( t_{k+1} - t \right)^2\,dt } 
+ \int_{t_{k}}^{t_{k+\frac12}}{\frac{(\zeta_2)_{ttt}(t)}{2 (\Delta t)^2} \left( t_{k} - t \right)^2\,dt } \nonumber \\
 &  & - A_2 \left( \int_{t_{k+ \frac12}}^{t_{k+1 }}{\frac{(\zeta_2)_{tt}(t)}{2 \Delta t} \left( t_{k+1} - t \right)\,dt } 
- \int_{t_{k}}^{t_{k+\frac12}}{\frac{(\zeta_2)_{tt}(t)}{2 \Delta t} \left( t_{k} - t \right)\,dt }\right)\nonumber \\ 
& & - b_2 \left( \int_{t_{k + \frac12}}^{t_{k+1}}{\frac{u_{ttt}(t,L)}{2 \Delta t} \left( t_{k+1} - t \right)\,dt }  - 
\int_{t_{k}}^{t_{k+\frac12}}{\frac{u_{ttt}(t, L)}{2 \Delta t} \left( t_{k} - t \right)\,dt }\right). \nonumber
\end{eqnarray} 
Using \eqref{CN__3}, \eqref{CN__4}, and \eqref{third_expansion}, we get
\begin{equation} \label{fd_formula_3}
\begin{array}{l}
\frac{\zeta^{k+1}_{e,1} - \zeta^{k}_{e,1}}{\Delta t} - A_1 \frac{\zeta^{k+1}_{e,1} + \zeta^{k}_{e,1}}{2} - b_1 
\frac{\Phi^{k+1}_x(L) + \Phi^{k}_x(L)}{2} = -\Delta t \, T_3^k - G^{k}_2, \\
\frac{\zeta^{k+1}_{e,2} - \zeta^{k}_{e,2}}{\Delta t} - A_2 \frac{\zeta^{k+1}_{e,2} + \zeta^{k}_{e,2}}{2} - b_2 
\frac{\Phi^{k+1}(L) + \Phi^{k}(L)}{2} = -\Delta t \, T_4^k - G^{k}_3,
\end{array}
\end{equation}
with
\begin{eqnarray*}
G_2^k & = &  b_1 \frac{u^e_{tx}(t_{k+1}, L) + u^e_{tx}(t_{k}, L)}{2},\\
G_3^k & = & b_2 \frac{u^e_{t}(t_{k+1}, L) + u^e_{t}(t_{k}, L)}{2}.
\end{eqnarray*} 
In \eqref{fd_formula_2} we now take  $w_h := \Delta t \frac{\Phi^{k+1} + \Phi^{k}}{2} \in W_h$, due to \eqref{fd_formula_1}. 
Using \eqref{fd_norm_part1} and \eqref{fd_formula_3}, yields:
\begin{eqnarray} \label{fd_norm_part2}
\| z^{k+1}_e \|^2 - \| z^{k}_e \|^2 & = & - (\Delta t)^2 \frac{1}{2} \nonumber
\int_{0}^{L}{\Lambda \left( \epsilon^{k+1}_{xx} + \epsilon^{k}_{xx} \right) (T^k_1)_{xx} \;dx} + 
\frac{\Delta t}{2} G^k_1(\Phi^{k+1} + \Phi^{k}) \\ &-& (\Delta t)^2 \left( k_1
 \frac{\epsilon_x^{k+1}(L) + \epsilon_x^{k}(L)}{2} (T_1^k)_x(L)
+ k_2 \frac{\epsilon^{k+1}(L) + \epsilon^{k}(L)}{2} T_1^k(L) \right)\nonumber \\
& - &\frac{\Delta t}{2} \left( q_1 \frac{\zeta^{k+1}_{e,1} + \zeta^{k}_{e,1}}{2} + 
\tilde{\delta}_1 \frac{\Phi^{k+1}_x(L) + \Phi^{k}_x(L)}{2} \right)^2 \nonumber \\
& - & \Delta t \delta_1 \left( \frac{\Phi^{k+1}_x(L) + \Phi^{k}_x(L)}{2}\right)^2 - 
\Delta t \frac{\epsilon_1}{2} \frac{\zeta_{e,1}^{k+1} + \zeta_{e,1}^{k}}{2} \cdot P_1 
\frac{\zeta_{e,1}^{k+1} + \zeta_{e,1}^{k}}{2} \nonumber \\
& - & P_1 \frac{\zeta_{e,1}^{k+1} + \zeta_{e,1}^{k}}{2} \cdot \left( (\Delta t)^2 T^k_3 + \Delta t \, G^k_2\right)
\nonumber \\
& - &\frac{\Delta t}{2} \left( q_2 \frac{\zeta^{k+1}_{e,2} + \zeta^{k}_{e,2}}{2} + 
\tilde{\delta}_2 \frac{\Phi^{k+1}(L) + \Phi^{k}(L)}{2} \right)^2 \nonumber \\
& - & \Delta t \delta_2 \left( \frac{\Phi^{k+1}(L) + \Phi^{k}(L)}{2}\right)^2 - 
\Delta t \frac{\epsilon_2}{2} \frac{\zeta_{e,2}^{k+1} + \zeta_{e,2}^{k}}{2} \cdot P_2 
\frac{\zeta_{e,2}^{k+1} + \zeta_{e,2}^{k}}{2} \nonumber \\
& - & P_2 \frac{\zeta_{e,2}^{k+1} + \zeta_{e,2}^{k}}{2} \cdot \left( (\Delta t)^2 T^k_4 + \Delta t \, G^k_3\right)
 \nonumber \\ &-& \frac12 (\Delta t)^2 T^k_2(\Phi^{k+1} + \Phi^{k}). \nonumber
\end{eqnarray}
Therefore,
\begin{eqnarray} \label{fd_norm_part3}
\| z^{k+1}_e \|^2 - \| z^{k}_e \|^2 & \le & - (\Delta t)^2 \frac{1}{2} \nonumber
\int_{0}^{L}{\Lambda \left( \epsilon^{k+1}_{xx} + \epsilon^{k}_{xx} \right) (T^k_1)_{xx} \;dx} + 
\frac{\Delta t}{2} G^k_1(\Phi^{k+1} + \Phi^{k}) \\ &-& (\Delta t)^2 \left( k_1
 \frac{\epsilon_x^{k+1}(L) + \epsilon_x^{k}(L)}{2} (T_1^k)_x(L)
+ k_2 \frac{\epsilon^{k+1}(L) + \epsilon^{k}(L)}{2} T_1^k(L) \right)\nonumber \\
& - & P_1 \frac{\zeta_{e,1}^{k+1} + \zeta_{e,1}^{k}}{2} \cdot \left( (\Delta t)^2 T^k_3 + \Delta t \, G^k_2\right)
\nonumber \\
& - & P_2 \frac{\zeta_{e,2}^{k+1} + \zeta_{e,2}^{k}}{2} \cdot \left( (\Delta t)^2 T^k_4 + \Delta t \, G^k_3\right)
 \nonumber \\ &-& \frac12 (\Delta t)^2 T^k_2(\Phi^{k+1} + \Phi^{k}).
\end{eqnarray}
Next, from \eqref{G_n1} follows:
\begin{eqnarray} 
|G^{k}_1(\Phi^{k+1} + \Phi^k)| &\le& C \left( \|
\frac{u^e_t(t_{k+1},x) - u^e_t(t_{k},x)}{\Delta t} \|_{L^2}^2 + \| \Phi^{k+1} + \Phi^k \|_{L^2}^2  \right.\nonumber \\ 
& +& | \frac{u^e_t(t_{k+1}, L) - u^e_t(t_{k}, L)}{\Delta t}|^2  
+   | \frac{u^e_{tx}(t_{k+1}, L) - u^e_{tx}(t_{k}, L)}{\Delta t}|^2  \nonumber \\  
 &+ & |\frac{u^e_{tx}(t_{k+1}, L) + u^e_{tx}(t_{k}, L)}{2}|^2 +  |\frac{u^e_{t}(t_{k+1}, L) +
 u^e_{t}(t_{k}, L) }{2}|^2 \nonumber\\ &+& \left. | \Phi^{k+1}(L) + \Phi^k(L)|^2 + | \Phi_x^{k+1}(L) + \Phi_x^{k}(L)|^2 
\vphantom{ \|\frac{u^e_t(t_{k+1},x) - u^e_t(t_{k},x)}{\Delta t} \|_{L^2}^2 }\right) 
\end{eqnarray}
\begin{eqnarray} \label{G_1 estimate}
 &\le& C \left( \| \Phi^{k+1} + \Phi^k \|_{L^2}^2 + 
| \Phi^{k+1}(L) + \Phi^k(L)|^2 + | \Phi_x^{k+1}(L) + \Phi_x^{k}(L)|^2 \right.\nonumber \\ 
& +& \frac{1}{\Delta t} \int_{t_k}^{t_k+1}{\| u^e_{tt}(t) \|^2_{L^2} + |u^e_{tt}(t,L)|^2 +|u^e_{ttx}(t,L)|^2 \, dt}
+  \|u^e_t \|_{C([t_{k},t_{k+1}]; H^2)}^2 \left. \right). \nonumber \\ 
\end{eqnarray}
It can easily be seen that 
\begin{equation} \label{T_1 estimate}
 \|T^k_1\|^2_{H^2} \le \Delta t \int_{t_k}^{t_{k+1}}{\|\breve u_{ttt}(t) \|^2_{H^2} \, dt} \le C 
\Delta t \int_{t_k}^{t_{k+1}}{\| u_{ttt}(t) \|^2_{H^2} \, dt},
\end{equation}
\begin{equation} \label{T_3 estimate}
 \|T^k_3\|^2 \le  C 
\Delta t \int_{t_k}^{t_{k+1}}{\| u_{ttt}(t) \|^2_{H^2} + \|(\zeta_1)_{tt}\|^2 + \|(\zeta_1)_{ttt}\|^2 \, dt},
\end{equation}
\begin{equation} \label{T_4 estimate}
 \|T^k_4\|^2 \le C 
\Delta t \int_{t_k}^{t_{k+1}}{\| u_{ttt}(t)\|^2_{H^1} + \|(\zeta_2)_{tt}\|^2 + \|(\zeta_2)_{ttt}\|^2 \, dt},
\end{equation}
and
\begin{eqnarray} \label{T_2 estimate}
T^k_2(\Phi^k) & \le & C \left(  \| \Phi^k \|^2_{L^2} + | \Phi^k(L) |^2 + | \Phi^k_x(L) |^2 
+ \vphantom{\int_{t_k}^{t_{k+1}}{\|u_{tt} (t)\|_{H^4}^2}} \nonumber \right. \\
& + &\Delta t \int_{t_k}^{t_{k+1}}{\|u_{tt} (t)\|_{H^4}^2 + 
\|u_{ttt}(t) \|_{H^2}^2 + \|u_{tttt}(t) \|_{H^2}^2 \, dt} \nonumber \\
& + & \left. \Delta t \int_{t_k}^{t_{k+1}}{\|(\zeta_1)_{tt} (t)\|^2 + \|(\zeta_2)_{tt} (t)\|^2 \, dt} \right).
\end{eqnarray} 
{}For the above estimate, we rewrote the second term of $T^k_2(\Phi^k)$ in \eqref{T_2_def} as:
\begin{align}
&\int^{L}_{0}{\left( \int_{t_{k+\frac12}}^{t_{k+1}}{\frac{u_{ttxx}(t,x)}{2 \Delta t} \left( t_{k+1} - t \right)\,dt }  - 
\int_{t_{k}}^{t_{k+\frac12}}{\frac{u_{ttxx}(t,x)}{2 \Delta t} \left( t_{k} - t \right)\,dt } \right) \Phi^k_{xx} \,dx} \nonumber &\\
=  &\int_{t_{k+\frac12}}^{t_{k+1}}{\frac{t_{k+1} - t}{2 \Delta t} \left( u_{ttxx}(t,L) \Phi^k_x(L) -  u_{ttxxx}(t,L) \Phi^k(L) + 
\int_{0}^{L}{u_{ttxxxx}(t,x) \Phi^k \, dx}\right) 
\,dt} & \nonumber \\
- &\int_{t_{k}}^{t_{k+\frac12}}{\frac{t_{k} - t}{2 \Delta t} \left( u_{ttxx}(t,L) \Phi^k_x(L) -  u_{ttxxx}(t,L) \Phi^k(L) + 
\int_{0}^{L}{u_{ttxxxx}(t,x) \Phi^k \, dx}\right) 
\,dt}, & \nonumber
\end{align} using $\Phi^k(0) = \Phi^k_x(0) = 0$, and then the Sobolev embedding Theorem.
{}From \eqref{fd_norm_part2} -- \eqref{T_2 estimate}, now follows:
\begin{eqnarray} \label{fd_norm_part4}
\| z^{k+1}_e \|^2 - \| z^{k}_e \|^2 & \le & C \left( \Delta t (\| z^{k+1}_e \|^2 +  \| z^{k}_e \|^2)\nonumber 
+  \Delta t  \|u^e_t \|_{C([t_{k},t_{k+1}]; H^2)}^2   \vphantom{\int_{t_k}^{t_{k+1}}{\| u^e_{tt}(t) \|^2_{L^2} + 
|u^e_{tt}(t,L)|^2 +|u^e_{ttx}(t,L)|^2 \, dt}}  \right. \\ &+& \int_{t_k}^{t_{k+1}}{\| u^e_{tt}(t) \|^2_{L^2} + 
|u^e_{tt}(t,L)|^2 +|u^e_{ttx}(t,L)|^2 \, dt} \nonumber \\ & + & (\Delta t)^4 \sum_{i=1}^{2}{
\int_{t_k}^{t_{k+1}}{ \|(\zeta_i)_{tt}\|^2 + \|(\zeta_i)_{ttt}\|^2 \, dt}} \nonumber \\
& + & \left.(\Delta t)^4 \int_{t_k}^{t_{k+1}}{\|u_{tt} (t)\|_{H^4}^2 + 
\|u_{ttt}(t) \|_{H^2}^2 + \|u_{tttt}(t) \|_{H^2}^2 \, dt}  \right).\nonumber \\
\end{eqnarray}
Let now $n \in \{1, \dots,  S\}$. Assuming $\Delta t \le \frac{1}{2 C}$ (with $C$ from \eqref{fd_norm_part4}), 
and summing \eqref{fd_norm_part4} over $k \in \{0, \dots, n\}$, gives:
\begin{eqnarray} \label{fd_norm_part5}
\frac12 \| z^{n+1}_e \|^2 & \le & \frac32 \| z^{0}_e \|^2 + C \left( \Delta t \sum_{k=1}^{n}{\| z^{k}_e \|^2}\nonumber 
 +   \|u^e_t \|_{C([0, T]; H^2)}^2  +  \right. \| u^e_{tt} \|^2_{L^2(0,T; H^2)}  \nonumber \\ 
& + & (\Delta t)^4 \left[ 
\sum_{i = 1}^2{\|(\zeta_i)_{tt} (t)\|_{L^2(0, T; \mathbb{R}^n)}^2 + 
\|(\zeta_i)_{ttt} (t)\|_{L^2(0, T; \mathbb{R}^n)}^2}
\right.  \nonumber \\
& + & \left. \left. \vphantom{\sum_{i = 1}^2{\|(\zeta_i)_{tt} (t)\|_{L^2(0, T; \mathbb{R}^n)}^2 + 
\|(\zeta_i)_{ttt} (t)\|_{L^2(0, T; \mathbb{R}^n)}^2}}
\|u_{tt} (t)\|_{L^2(0, T; H^4)}^2 + 
\|u_{ttt}(t) \|_{L^2(0, T; H^2)}^2 + \|u_{tttt}(t) \|_{L^2(0, T; H^2)}^2 
\right] \right).\nonumber \\
\end{eqnarray}
Finally, using the discrete-in-time Gronwall inequality and \eqref{first_expansion}, we obtain:
\begin{eqnarray} \label{fd_norm_part6}
 \| z^{n+1}_e \|^2 & \le &  C \left( \| z^{0}_e \|^2 \nonumber 
\vphantom{\sum_{i = 1}^2{\|(\zeta_i)_{tt} (t)\|_{L^2(0, T; \mathbb{R}^n)}^2 + 
\|(\zeta_i)_{ttt} (t)\|_{L^2(0, T; \mathbb{R}^n)}^2}}
 +   h^4 \left( \|u_t \|_{C([0, T]; H^4)}^2  +   \| u_{tt} \|^2_{L^2(0,T; H^4)} \right) \right. \nonumber \\ 
& + & (\Delta t)^4 \left[ 
\sum_{i = 1}^2{\|(\zeta_i)_{tt} (t)\|_{L^2(0, T; \mathbb{R}^n)}^2 + 
\|(\zeta_i)_{ttt} (t)\|_{L^2(0, T; \mathbb{R}^n)}^2}
\right.  \nonumber \\
& + & \left. \left. \vphantom{\sum_{i = 1}^2{\|(\zeta_i)_{tt} (t)\|_{L^2(0, T; \mathbb{R}^n)}^2 + 
\|(\zeta_i)_{ttt} (t)\|_{L^2(0, T; \mathbb{R}^n)}^2}}
\|u_{tt} (t)\|_{L^2(0, T; H^4)}^2 + 
\|u_{ttt}(t) \|_{L^2(0, T; H^2)}^2 + \|u_{tttt}(t) \|_{L^2(0, T; H^2)}^2 
\right] \right).\nonumber \\
\end{eqnarray} 
The result now follows from \eqref{fd_norm_part6}, \eqref{projection_estimates}, and the triangle inequality. 
\end{proof}

\vspace{\baselineskip}

\end{document}